\documentclass{article}

\usepackage{amsmath,amssymb,amsthm,graphicx,proof,color}
\usepackage[UKenglish]{babel}

\usepackage{hyperref}
\usepackage[all]{xy}

\DeclareSymbolFont{symbolsC}{U}{txsyc}{m}{n}
\DeclareMathSymbol{\strictif}{\mathrel}{symbolsC}{74}
\renewcommand{\prec}{\strictif}

\newcommand{\bis}{\mathrel{\mathchoice%
{\raisebox{.3ex}{$\,
  \underline{\makebox[.7em]{$\leftrightarrow$}}\,$}}%
{\raisebox{.3ex}{$\,
  \underline{\makebox[.7em]{$\leftrightarrow$}}\,$}}%
{\raisebox{.2ex}{$\,
  \underline{\makebox[.5em]{\scriptsize$\leftrightarrow$}}\,$}}%
{\raisebox{.2ex}{$\,
  \underline{\makebox[.5em]{\scriptsize$\leftrightarrow$}}\,$}}}}

\renewcommand{\vartriangleleft}{\Rrightarrow}

\hypersetup{
   colorlinks,%
   citecolor=blue,%
   filecolor=black,%
   linkcolor=red,%
   urlcolor=black
 }

\usepackage[all]{xy}
\usepackage{tikz}
\usetikzlibrary{trees}
\usepgflibrary{arrows}
\usetikzlibrary{positioning,shapes}
\usetikzlibrary{patterns}

\newcommand{\M}{\mathcal{M}}

\newcommand{\TAUT}{\ensuremath{\texttt{TAUT}}}

\newcommand{\MP}{\ensuremath{\texttt{MP}}}

\newcommand{\BP}{\ensuremath{\textbf{P}}}

\newcommand{\SPLKw}{\ensuremath{\mathbb{CL}}}
\newcommand{\SPLKwT}{\ensuremath{\mathbb{CLT}}}
\newcommand{\SPLKwB}{\ensuremath{\mathbb{CLB}}}

\newcommand{\SPLKwTr}{\ensuremath{\mathbb{CL}4}}

\newcommand{\SPLKwEuc}{\ensuremath{\mathbb{CL}5}}

\newcommand{\SPLKwTrEuc}{\ensuremath{\mathbb{CL}45}}
\newcommand{\SPLKwTTr}{\ensuremath{\mathbb{CLS}4}}
\newcommand{\SPLKwTEuc}{\ensuremath{\mathbb{CLS}5}}

\newcommand{\lr}[1]{\langle #1 \rangle}
\newcommand{\lra}{\leftrightarrow}

\newcommand{\GEN}{\ensuremath{\texttt{GEN}\Delta}}
\newcommand{\REKw}{\ensuremath{\texttt{RE}\Delta}}
\newcommand{\KwCon}{\ensuremath{\Delta\texttt{Con}}}
\newcommand{\KwDis}{\ensuremath{\Delta\texttt{Dis}}}
\newcommand{\EquiKw}{\ensuremath{{\Delta\texttt{Equ}}}}

\newcommand{\KwT}{\ensuremath{\Delta\texttt{T}}}
\newcommand{\KwTr}{\ensuremath{\Delta\texttt{4}}}
\newcommand{\KwEuc}{\ensuremath{\Delta\texttt{5}}}
\newcommand{\KwB}{\ensuremath{\Delta\texttt{B}}}
\newcommand{\Tr}{\ensuremath{\texttt{w}\Delta4}}
\newcommand{\Euc}{\ensuremath{\texttt{w}\Delta5}}

\renewcommand{\phi}{\varphi}
\renewcommand{\sharp}{\#}

\newtheorem{theorem}{Theorem}
\newtheorem{lemma}[theorem]{Lemma}
\newtheorem{definition}[theorem]{Definition}
\newtheorem{proposition}[theorem]{Proposition}

\newtheorem{corollary}[theorem]{Corollary}

\newtheorem{fact}[theorem]{Fact}

\newtheorem{conjecture}[theorem]{Conjecture}

\newcommand{\weg}[1]{}

\title{A Modal Logic of Supervenience}

\author{Jie Fan\\
\small School of Philosophy, Beijing Normal University  \\
\small \texttt{fanjie@bnu.edu.cn}}
\date{26th July 2016 --- \today}

\begin{document}
\maketitle

\begin{abstract}
Supervenience is an important philosophical concept. In this paper, inspired by the supervenience-determined consequence relation and the semantics of agreement operator, we introduce a modal logic of supervenience, which has a dyadic operator of supervenience as a sole modality. The semantics of supervenience modality is very natural to correspond to the supervenience-determined consequence relation, in a quite similar way that the strict implication corresponds to the inference-determined consequence relation. We show that this new logic is more expressive than the modal logic of agreement, by proposing a notion of bisimulation for the latter logic. We provide a sound proof system for our new logic. We also lift on to more general logics of supervenience. Related to this, we compare propositional logic of determinacy and non-contingency in expressive powers, and give axiomatizations of propositional logic of determinacy over various classes of frames, thereby resolving an open research direction listed in~\cite[Sec.~8.2]{Gorankoetal:2016}. As a corollary, we also present an alternative axiomatization for propositional logic of determinacy over universal models. We conclude with a lot of future work.
\end{abstract}


\section{Introduction}
\subsection{Philosophical motivation}
Past decades witness an increasing interest in the concept of {\em supervenience}, which has traditionally been used as a relation between sets of properties.\footnote{There are some exceptions, though. For instance, in~\cite{Hellmanetal:1975},~\cite{Haugeland1982-HAUWS} and~\cite{Davidson:1985}, supervenience is a relation between sets of predicates, between a pair of languages, and between a predicate and a set of predicates in a language, respectively.}\weg{The notion of supervenience is assumed to be a relation between two families of properties.} A set $A$ of properties (called `supervenient properties') is said to {\em supervene on} another set $B$ (called `subveinent properties'), just in case if $B$-properties are indistinguishable, then so are $A$-properties; in other words, agreement in respect of $B$-properties implies agreement in respect of $A$-properties. In slogan form, ``there cannot be an $A$-different without a $B$-difference''~\cite{McLaughlinetal-supervenience}. \weg{Usually, $A$ and $B$ are called `the supervenient class' and `the subvenient class', respectively.}The core idea of supervenience is that fixing subvenient properties fixes its supervenient ones; or equivalently, subvenient properties {\em determine} supervenient properties.

The notion of supervenience dates back at least to\weg{generally acknowledged to be traceable to} G.~E.~Moore's classical work~\cite{Moore:1922}, where he described some certain dependency relationship between moral and non-moral properties\weg{it is argued that moral properties supervene upon non-moral properties}\weg{some certain dependency relationship between moral and non-moral properties are described that has
later come to be called "supervenience"}. However, Moore did not use the term `supervenience' explicitly;\weg{it was R.~M.~Hare who coined the term `supervenience' in~\cite{Hare:1952}.} it was R.~M.~Hare~\cite{Hare:1952} that introduced the term into the philosophical literature, to characterize a relationship between moral properties and natural properties.\footnote{Hare~\cite[p.~145]{Hare:1952} stated ``First, let us take that characteristic of `good' which has been called its supervenience. Suppose that we say `St. Francis was a good man'. It is logically impossible to say this and to maintain at the same time that there might have been another man placed in precisely the same circumstances as St. Francis, and who behaved in them in exactly the same way, but who differed from St. Francis in this respect only, that he was not a good man.'' Similar idea can be identified in other places in that book, e.g.~pp.~80-81, p.~134, p.~153. Here, Hare spoke of supervenience as a characteristic of the term `good'. But as~\cite[p.~155]{Kim:1984} commented, it is better to think of supervenience as a relation between the property of being a good man and the properties such as patterns of behavior and traits of character.} Thanks to Donald Davidson~\cite{Davidson:1970}, the term `supervenience' was first introduced into contemporary philosophy of mind,\footnote{Donald Davidson used psychophysical supervenience to defend a position of anomalous monism that although the mental supervenes on the physical, the former cannot be reduced to the latter, as he said on~\cite[p.~88]{Davidson:1970}: ``Although the position I describe denies there are psychophysical laws, it is consistent with the view that mental characteristics are in some sense dependent, or supervenient, on physical characteristics. Such supervenience might be taken to mean that there cannot be two events alike in all physical respects but differing in some mental respect, or that an object cannot alter in some mental respect without altering in some physical respect. Dependence or supervenience of this kind does
not entail reducibility through law or definition $\cdots$.'' Such a supervenience thesis was explicitly advocated on Davidson~\cite[pp. 716-717]{Davidson:1973}.} which opened up a new research direction in this area and other branches of philosophy\weg{a research interest in supervenience in its own right}, see e.g.~\cite{Hellmanetal:1975,Kim:1978,Kim:1979,Horgan:1981,Haugeland1982-HAUWS,Horgan:1982,Kim:1982,Lewis:1983,Horgan:1984,Horgan:1993,Leuenberger:2009}. It is alleged (e.g.~\cite{McLaughlinetal-supervenience}) that every major figure in the history of western philosophy has been at least implicitly committed to some supervenience thesis. For example,  Leibniz used the Latin word `supervenire', to state the thesis that relations are supervenient on properties (e.g.~\cite{Ishiguro:1972}); G.~E.~Moore stated that ``one of the most important facts about qualitative difference~$\cdots$ [is that] two things cannot differ in quality without differing in intrinsic nature''~(\cite[p.~263]{Moore:1922}); David Lewis used a thesis of Humean supervenience to express that the whole truth about a world like ours supervenes on the spatiotemporal distribution of local qualities~\cite[pp.~ix--xvi]{Lewis:1986a}. 

The notion of supervenience is ubiquitous in our daily life. For instance, the aesthetic properties of a work of art supervene on its physical properties, the price of a commodity supervenes on its supply and demand, effects supervene on causes, and the mental supervenes on the physical. According to the chart of levels of existence~\cite{Supervenience-wiki}, atoms supervene on elementary particles, molecules supervene on atoms, cells supervene on molecules, and so on. Moreover, a number of interesting doctrines and problems can be formulated in terms of supervenience. A paradigmatic example is physicalism, which may be construed as a thesis that ``everything supervenes on the physical''~\cite{Stoljar:2015}. Mereology may be explained as mereological supervenience, i.e., the whole supervenes on its parts~(cf.~e.g.~\cite[p.~101]{Kim1984-KIMEAS}). Determinism can be roughly construed as a thesis that everything to the future supervenes on the present, and perhaps past, facts. All of the distinction between internalism and externalism can be characterized by means of supervenience theses~\cite{McLaughlinetal-supervenience}. Mind-body problem may be rephrased as to whether the psychophysical supervenience thesis holds, i.e., are psychological properties supervenient upon physical properties (e.g.~\cite{Davidson:1970,Kim:1979,Kim:1982,Kim1982-KIMPSA})?

There are so many distinct formulations for this concept, e.g., individual supervenience, local supervenience, global supervenience, weak supervenience, strong supervenience, similarity-based supervenience, regional supervenience, local-local supervenience and strong-local-local supervenience, multiple domain supervenience (c.f. e.g. \cite{Kim:1984,Kim:1987,Horgan:1982,Haugeland1982-HAUWS,Hofweber:2005,Kim:1988}),
that David Lewis thought of it as an `unlovely proliferation'~\cite[p.~14]{Lewis:1986}. No matter how different the formulations are, they all conform to the aforementioned core idea of supervenience --- that is, fixing the subvenient properties fixes the supervenient properties. 

Supervenience has many applications, among which a central use is so-called `argument by a false implied supervenience thesis'. It is well known that the reduction of A to B implies the supervenience of A on B; in short, reduction implies supervenience. Thus for one to argue against a reduction thesis, it suffices to falsify the corresponding supervenience thesis. Other applications include characterizing the distinctions between Internalism and Externalism, characterizing physicalism, characterizing haecceitism, and so on. For the details of all these applications, a highly recommendation would be~\cite{McLaughlinetal-supervenience}. 

In spite of so many philosophical discussions for the notion of supervenience, there have been few studies in the sphere of logic. The only logical work dealing with supervenience that we have found in the literature are a series of publications written by Humberstone~\cite{Humberstone:1992,Humberstone:1993,Humberstone:1996,Humberstone:1998}, all of which are in terms of valuations/truth assignments. The related notions of supervenience are contingency, agreement, and dependence/determination, see~\cite{Fanetal:2015,Humberstone:2002,Gorankoetal:2016} and references thereof.\footnote{It was argued in~\cite{Rickles:2006} that supervenience is the converse concept of determination. That is, to say that a set $A$ of properties supervenes on another set $B$, is equivalent to say that $B$ determines $A$.} We will propose a modal logic of supervenience, and compare our logic with these related logics.

In this paper, unlike Humberstone's method, we will treat the notion of supervenience as a primitive modality, based on possible worlds rather than valuations. This idea seems very natural, since for instance, in an oft-cited work~\cite{Horgan:1993}, Horgan claimed, ``Supervenience, then, is a modal notion.'' (p.~555). 
Besides, instead of exploring supervenience for properties, we investigate the supervenience relation between two (sets of) sentences/formulas, which can be justified by Kim's claim ``One could also speak of supervenience for {\em sentences}, facts, events, propositions, and languages''~\cite[p.~155]{Kim:1984}. We will demonstrate that the sentence supervenience has many similar results to the property supervenience.


\subsection{Technical motivation}

Technically, our paper is mainly inspired by a notion of supervenience-determined consequence relation, together with the semantics of an agreement operator in the literature.

Humberstone~\cite{Humberstone:1993} distinguished between two types of consequence relations: inference-determined and supervenience-determined. Inference-determined version is just Tarski's consequence relation, i.e., given a class of valuations $\mathcal{V}$, the inference-determined consequence relation by $\mathcal{V}$, denoted $\vDash_\mathcal{V}$, is defined as:
\[
\begin{array}{lcl}
\Gamma\vDash_\mathcal{V} A&\iff& \text{for all valuations }v\in \mathcal{V},\\
&&\text{if }v(B)=T\text{ for each }B\in\Gamma, \text{ then }v(A)=T.\\
\end{array}
\]
In comparison, the consequence relation {\em supervenience-determined} by $\mathcal{V}$, denoted $\Vdash_\mathcal{V}$, is defined as:
\[
\begin{array}{lcl}
\Gamma\Vdash_\mathcal{V} A&\iff& \text{for all valuations }u,v\in \mathcal{V},\\
&&\text{if }u(B)=v(B)\text{ for each }B\in\Gamma, \text{ then }u(A)=v(A).\\
\end{array}
\]

\weg{Humberstone~\cite{Humberstone:1993} differentiated two consequence relations: inference-determined and supervenience-determined. Inference-determined consequence relation is just the standard definition for Tarski's consequence relation, i.e., given a class of valuations $\mathcal{V}$, the inference-determined consequence relation by $\mathcal{V}$, denoted $\vDash_\mathcal{V}$, is defined as:
\[
\begin{array}{lcl}
\{A_1,\cdots,A_n\}\vDash_\mathcal{V} B&\iff& \text{for all valuations }v\in \mathcal{V},\\
&&\text{if }v(A_i)=T\text{ for each }A_i, \text{ then }v(B)=T.\\
\end{array}
\]
In comparison, the consequence relation {\em supervenience-determined} by $\mathcal{V}$, denoted $\Vdash_\mathcal{V}$, is defined as:
\[
\begin{array}{lcl}
\{A_1,\cdots,A_n\}\Vdash_\mathcal{V} B&\iff& \text{for all valuations }u,v\in \mathcal{V},\\
&&\text{if }u(A_i)=v(A_i)\text{ for each }A_i, \text{ then }u(B)=v(B).\\
\end{array}
\]}

Inspired by the supervenience-determined consequence relation, in a rather natural sense, we introduce a dyadic operator $\vartriangleleft$ and interpret it on a Kripke model $\M$ with a domain $W$ as follows:
\[\begin{array}{lclr}
\M,w\vDash B\vartriangleleft A&\iff&\text{for all }u,v\in W, \text{ if }(\M,u\vDash B\iff \M,v\vDash B),& \\
&&\text{then }(\M,u\vDash A\iff \M,v\vDash A).&(\text{Def~1})
\end{array}\]

This definition is in line with the supervenience-determined consequence relation outlined above. Note that we now take $u$ and $v$ as possible worlds rather than valuations. An obvious difference between possible worlds and valuations is, whereas there can be two distinct possible worlds in a model which agree on all formulas, that would not be so for valuations.

But note that the modality $\vartriangleleft$ is global but not local, in the sense that its truth does not depend on the designated state where it is evaluated. An equivalent saying for this is that $\vartriangleleft$ is defined on a {\em universal} model.

There are also alternative variations for (Def~1). Recall that the agreement operator (denoted $O$) is interpreted on a generalized model $\M=\lr{W,S,V}$, where $W$ and $V$ are as usual, and $S$ is a ternary relation without any constraints (for more details, see~\cite[p.~107]{Humberstone:2002}, or Section~\ref{subsec.agreement}).
\[\begin{array}{lcl}
\M,w\vDash OA&\iff&\text{for all }u,v\in W\text{ such that }S_wuv, \\
&&\text{we have }(\M,u\vDash A\iff \M,v\vDash A).
\end{array}\]

Now inspired by the semantics of agreement operator, we add the premise $S_wuv$ into the right-hand side of (Def~1), thereby obtaining a much more general semantics for $\vartriangleleft$:
\[\begin{array}{lclr}
\M,w\vDash B\vartriangleleft A&\iff& \text{for all }u,v\in \M\text{ such that }S_wuv,&\\
&&\text{if }(\M,u\vDash B\iff \M,v\vDash B),&\\
&&\text{then }(\M,u\vDash A\iff\M,v\vDash A).&(\text{Def~2})\\
\end{array}\]
The modality $\vartriangleleft$ in (Def~2) is now local, which takes the global version (Def~1) as a special case when the accessibility relation $S$ is universal, in the sense that for all $w,u,v\in W$, it holds that $S_wuv$.

\medskip

The reminder is organized as follows. After briefly reviewing the related logics in the literature (Sec.~\ref{sec.preliminaries}), we introduce a modal logic $\mathcal{L}_\vartriangleleft$ of supervenience, defining its language and semantics (Sec.~\ref{sec.supervenience}). Then we compare the relative expressive powers of $\mathcal{L}_\vartriangleleft$ and the modal logic $\mathcal{L}_O$ of agreement, by proposing a bisimulation notion for the latter (Sec.~\ref{sec.exp-ls-lo}). Besides providing a sound proof system for $\mathcal{L}_\vartriangleleft$ (Sec.~\ref{sec.soundsystem}), we lift the new logic on to more general logics of supervenience in Sec.~\ref{sec.generalizedsup}, depending on the arity of the concerned supervenience operators, which take propositional logic $\mathcal{L}_D$ of determinacy as a special case when the underlying accessibility relation is universal. Sec.~\ref{sec.generalizedsup} contains two main results: Sec.~\ref{sec.exp-lcld} compares the expressive power of contingency logic $\mathcal{L}_\Delta$ and $\mathcal{L}_D$, where it turns out that they are equally expressive over the class of all models; in Sec.~\ref{sec.axiomatization-ld}, we give axiomatizations for $\mathcal{L}_D$ over various classes of frames, via a reduction to the completeness of their based axiomatizations for $\mathcal{L}_\Delta$. The two results obtained in Sec.~\ref{sec.generalizedsup}, we think, complete an open research direction listed in~\cite{Gorankoetal:2016}. As a corollary, we present an alternative axiomatization for $\mathcal{L}_D$ over universal models. We conclude with a lot of future work in Sec.~\ref{sec.conclusions}.

In summary, our contributions consist of:
\begin{itemize}
\item A semantics of the supervenience operator~(Sec.~\ref{sec.supervenience}) and a sound proof system for the modal logic $\mathcal{L}_\vartriangleleft$ of supervenience~(Sec.~\ref{sec.soundsystem});
\item A bisimulation notion for the modal logic $\mathcal{L}_O$ of agreement (Sec.~\ref{sec.exp-ls-lo});
\item $\mathcal{L}_\vartriangleleft$ is more expressive than $\mathcal{L}_O$ over the class of all models~(Sec.~\ref{sec.exp-ls-lo});
\item Contingency logic is equally expressive as propositional logic $\mathcal{L}_D$ of determinacy over the class of all models (Sec.~\ref{sec.exp-lcld});
\item Various axiomatizations for $\mathcal{L}_D$ over various frame classes, and completeness proof via a reduction method (Sec.~\ref{sec.axiomatization-ld});
\item An alternative axiomatization for $\mathcal{L}_D$ over universal models (Sec.~\ref{sec.axiomatization-ld}).
\end{itemize}

\section{Preliminaries}\label{sec.preliminaries}

Throughout this paper, we let $\BP$ denote the set of proposition symbols, and let $p$ denote an element of $\BP$. Also, we use the notation $\mathcal{D}$, $\mathcal{T}$, $4$, $5$, $\mathcal{B}$, $\mathcal{S}4$, $\mathcal{S}5$, $45$, and $\mathcal{KD}45$ to stand for, respectively, the class of all serial frames, the class of all reflexive frames, the class of all transitive frames, the class of all Euclidean frames, the class of all symmetric frames, the class of all reflexive and transitive frames, the class of reflexive and Euclidean frame, the class of transitive and Euclidean frames, and the class of all serial, transitive and Euclidean frames.

\subsection{Contingency logic}\label{sec.contingency}

Contingency logic is an extension of propositional logic with a primitive modality $\Delta$. In symbol, contingency logic $\mathcal{L}_\Delta$ is defined inductively as the following BNF:
$$A::=p\mid \neg A\mid A\land A\mid \Delta A$$
Where $\Delta A$ is read ``it is {\em non-contingent} that $A$''. Boolean connectives are interpreted as usual, and the non-contingency operator $\Delta$ is interpreted by the following semantics: given a Kripke model $\M=\lr{W,R,V}$ and a world $w\in W$,
\[
\begin{array}{lll}
\M,w\vDash\Delta A&\iff &\text{for all }u,v\in W\text{ such that }wRu\text{ and }wRv, \\
&&\text{we have }(\M,u\vDash A\iff \M,v\vDash A).\\
\end{array}
\]
It is known that $\Delta$ is definable in terms of the necessity operator $\Box$, as $\Delta A=_{df}\Box A\lor\Box\neg A$, see e.g.~\cite{MR66}. However, $\Delta$ has some advantages over $\Box$, for example, as the definition itself indicates, the statements with $\Delta$ is exponentially more succinct than those with $\Box$ instead of $\Delta$ which are intended to express the same propositions~\cite{Hansetal:2014}. The modality $\Delta$ also invites novel techniques in axiomatizing uni-modal and poly-modal $\mathcal{L}_\Delta$ over various frame classes, see~\cite{Humberstone95,DBLP:journals/ndjfl/Kuhn95,DBLP:journals/ndjfl/Zolin99,hoeketal:2004,Fanetal:2014,Fanetal:2015}. A variety of axiomatizations have been proposed in the literature, we here adopt the axiomatizations of~\cite{Fanetal:2015} for the sake of reference in Sec.~\ref{sec.axiomatization-ld}.\footnote{As shown in~\cite[pp.~110-111]{Humberstone:2002}, the axiom $\KwDis$ can be replaced with $\Delta A\to \Delta(A\to D)\vee \Delta(\neg A\to D)$, which is denoted by {\bf H}. The proof is as follows: firstly, {\bf H} is obviously a special instance of $\KwDis$ when $B=C=D$. Conversely, let $D$ be $(A\to B)\land(\neg A\to C)$, then it is easy to show that $(A\to B)\lra(A\to(A\to B)\land(\neg A\to C))$ and $(\neg A\to C)\lra(\neg A\to(A\to B)\land(\neg A\to C))$, i.e. $(A\to B)\lra(A\to D)$ and $(\neg A\to C)\lra(\neg A\to D)$, respectively, where $B,C$ are arbitrary. Then using the rule $\REKw$, we infer the axiom $\KwDis$ from {\bf H}.}

\begin{center}
\begin{tabular}{lclc}
\multicolumn{4}{c}{System \SPLKw}\\
\multicolumn{2}{l}{Axiom Schemas}&Rules&\\
\TAUT & \text{all instances of tautologies} & \MP & $\dfrac{A,A\to B}{B}$ \\
\KwCon & $\Delta(A\to B)\land\Delta(\neg A\to B)\to\Delta B$&  \GEN & $\dfrac{A}{\Delta A}$\\
\KwDis & $\Delta A\to \Delta (A\to B)\lor \Delta(\neg A\to C)$ &\REKw & $\dfrac{A\lra B}{\Delta A\lra\Delta B}$ \\
\EquiKw & $\Delta A\lra\Delta\neg A$ && \\
\end{tabular}
\end{center}

\[ \begin{array}{|l|l|l|l|}
  \hline
  \text{Notation}& \text{Axiom Schemas}& \text{Systems} & \text{Frames} \\
\hline
  && \SPLKw & {\mathcal D}\\
  \hline
  \KwT & \Delta A\land\Delta(A\to B)\land A\to\Delta B& \SPLKwT=\SPLKw+\KwT & {\mathcal T} \\
  \KwTr & \Delta A\to\Delta(\Delta A\vee B)& \SPLKwTr=\SPLKw+\KwTr & 4\\
  \KwEuc & \neg\Delta A\to\Delta(\neg\Delta A\vee B)&\SPLKwEuc=\SPLKw+\KwEuc & 5\\
  \KwB   & A\to\Delta((\Delta A\land\Delta(A\to B)\land\neg\Delta B)\to C)& \SPLKwB=\SPLKw+\KwB& {\mathcal B}\\
   \Tr  & \Delta A\to\Delta\Delta A &\SPLKwTTr=\SPLKw+\KwT+\Tr & {\mathcal S}4\\
   \Euc   & \neg\Delta A\to\Delta\neg\Delta A &\SPLKwTEuc=\SPLKw+\KwT+\Euc & {\mathcal S}5 \\
   \hline
   &&\SPLKwTrEuc=\SPLKw+\KwTr+\KwEuc&45(\mathcal{KD}45)\\
  \hline
\end{array}
\]

\begin{theorem}(c.f.~\cite{Fanetal:2014,Fanetal:2015})\label{thm.com-lc}
$\SPLKw$ is sound and strongly complete with respect to the class of all frames and also the class of all $\mathcal{D}$-frames, other extensions of $\SPLKw$ are sound and strongly complete with respect to the corresponding class of frames listed in the above table.
\end{theorem}

\subsection{The modal logic of agreement}\label{subsec.agreement}

Humberstone~\cite{Humberstone:2002} proposed a so-called `the modal logic of agreement', to lift the study of contingency logic on to a general modal logic. The modal logic of agreement $\mathcal{L}_O$ extends propositional logic with an operator of agreement $O$ rather than $\Delta$:
$$A::=p\mid \neg A\mid A\land A\mid OA.$$

The model $\M$, called `a generalized model' in~\cite{Humberstone:2002}, is a triple $\lr{W,S,V}$, where $W$ and $V$ are as usual, i.e., $W$ is a set of possible worlds and $V$ a valuation from $\BP$ to $\mathcal{P}(W)$, and $S$ is a ternary relation without any constraints. The agreement operator $O$ is interpreted as follows:
\[\begin{array}{lcl}
\M,w\vDash OA&\iff&\text{for all }u,v\in W\text{ such that }S_wuv, \\
&&\text{we have }(\M,u\vDash A\iff \M,v\vDash A).
\end{array}\]
Intuitively, $S_wuv$ means that $u$ and $v$ stand in the relation that is the value of $S$ for the argument $w$. When $S$ is defined such that, for any $w,u,v\in W$,  $S_wuv$ holds just in case $wRu$ and $wRv$, we obtain the interpretation of non-contingency operator $\Delta$, and thus the interpretation of $\Delta$ is a special case of that of $O$. Other special cases for the semantics of $O$ refer to~\cite[Sec.~3]{Humberstone:2002}.

In the sequel, partly inspired by the work of the modal logic of agreement, we will introduce one of its extensions, called `a logic of supervenience', in which our supervenience operator is defined on the generalized models. We will show that our new logic is more expressive than $\mathcal{L}_O$, by defining a bisimulation notion of $\mathcal{L}_O$.

\subsection{Propositional logic of determinacy}\label{sec.determinacy}

Propositional logic of determinacy $\mathcal{L}_D$ is introduced in~\cite{Gorankoetal:2016}, a logic which extends propositional logic with dependence formulas $D(A_1,\cdots,A_n;B)$, where $A_1,\cdots, A_n,B$ are all arbitrary formulas in $\mathcal{L}_D$. In symbol,
$$A::=p\mid \neg A\mid A\land A\mid D(A,\cdots,A;A)$$
Where the tuple $(A,\cdots,A;A)$ contains $n+1$ formulas for any $n\in\mathbb{N}$.

Formula $D(A_1,\cdots,A_n;B)$ is read ``$B$ depends only on $A_1,\cdots,A_n$'', intuitively meaning that the truth value of $B$ is determined by the set of truth values of $A_1,\cdots,A_n$; or roughly speaking, once the truth value of each $A_i$ ($i\in[1,n]$) are fixed, the truth value of $B$ is also fixed. The determinacy operator $D$ is interpreted in~\cite{Gorankoetal:2016} on Kripke models $\M=\lr{W,R,V}$, where $W$ and $V$ are as usual, and $R$ is the {\em universal} relation, i.e. for all $w,v\in W$, $wRv$.
\[
\begin{array}{lll}
\M,w\vDash D(A_1,\cdots,A_n;B)&\iff&\text{for all }u,v\in W\text{ such that }wRu\text{ and }wRv,\\
&& \text{if }(\M,u\vDash A_i\iff \M,v\vDash A_i)\text{ holds for all }i\leq n,\\
&&\text{then }(\M,u\vDash B\iff \M,v\vDash B).\\
\end{array}
\]

At the end of that paper~\cite[Sec.~8.2]{Gorankoetal:2016}, the semantic of $D$ was generalized into the general modal setting, without any constraints for the accessibility relation $R$. It left as an open research direction how to investigate the determinacy operator $D$ over various classes of Kripke models. In this article, we will complete this research direction, by showing that $D$ is inter-definable with the non-contingency operator $\Delta$ on all Kripke models, and we will also give axiomatizations of $\mathcal{L}_D$ over various frame classes.

\section{Supervenience logic}\label{sec.supervenience}

In this section, we introduce our modal logic of supervenience. Before doing that, let us display an interesting contrast between the semantics of our new modality with that of strict implication.

Being unhappy with a so-called `paradoxes of material implication', in his seminal work~\cite{Lewis:1918}, Lewis defined a strict implication $\prec$, with $A\prec B$ read as ``$A$ strictly implies $B$'' and interpreted by the following:
\[\begin{array}{lcl}
\M,w\vDash A\prec B&\iff& \text{for all }u\in \M\text{ such that }Rwu,\text{ if }\M,u\vDash A,\text{ then }\M,u\vDash B.
\end{array}\]
The operator $\prec$ internalizes the claim for $B$'s being a consequence, by some inference-determined consequence relation, of a formula $A$. Recall that given a class of valuations $\mathcal{V}$, the consequence relation {\em inference-determined} by $\mathcal{V}$, denoted $\vDash_\mathcal{V}$, defined as:
\[
\begin{array}{lcl}
\Gamma\vDash_\mathcal{V} A&\iff& \text{for all valuations }v\in \mathcal{V},\\
&&\text{if }v(B)=T\text{ for each }B\in\Gamma, \text{ then }v(A)=T.\\
\end{array}
\]
In comparison, the consequence relation {\em supervenience-determined} by $\mathcal{V}$~\cite{Humberstone:1993}, denoted $\Vdash_\mathcal{V}$, is defined as:
\[
\begin{array}{lcl}
\Gamma\Vdash_\mathcal{V} A&\iff& \text{for all valuations }u,v\in \mathcal{V},\\
&&\text{if }u(B)=v(B)\text{ for each }B\in\Gamma, \text{ then }u(A)=v(A).\\
\end{array}
\]
Then it is very natural to define a binary operator $\vartriangleleft$ as follows:
\[\begin{array}{lcl}
\M,w\vDash A\vartriangleleft B&\iff& \text{for all }u,v\in \M\text{ such that }S_wuv,\\
&&\text{if }\M,u\vDash A\iff \M,v\vDash A,\\
&&\text{then }\M,u\vDash B\iff\M,v\vDash B.\\
\end{array}\]

We thus obtain our supervenience operator. The operator $\vartriangleleft$ internalizes the claim that $B$ is a consequence by some supervenience-determined consequence relation of $A$. Roughly speaking, just as strict implication is a localized object-language modality corresponding to inference-determined consequence (in the 1-premiss case), the supervenience operator is a localized object-language modality corresponding to supervenience-determined consequence (in the 1-premiss case).

\begin{definition}The language $\mathcal{L}_\vartriangleleft$ of {\em the modal logic of supervenience} is defined inductively by the following BNF:
$$A::=p\mid \neg A\mid (A\land A)\mid (A\vartriangleleft A)$$
\end{definition}
Thus $\mathcal{L}_\vartriangleleft$ is an extension of propositional logic with a new dyadic modality $\vartriangleleft$. The construct $A\vartriangleleft B$ is read ``$B$ supervenes on $A$''. Intuitively, $B$ supervenes on $A$, if (given some premise,) once the truth value of $A$ is fixed, the truth value of $B$ is also fixed. We always drop the parentheses around formulas whenever no confusion arises.

A {\em model of $\mathcal{L}_\vartriangleleft$} is a triple $\M=\lr{W,S,V}$, where $W$ is a nonempty set of worlds, $V:\BP\to\mathcal{P}(W)$ is a valuation, and $S$ a function assigning to each $w\in W$ a binary relation $S_w$ on $W$. We can equally think of $S$ as an arbitrary ternary relation on $W$. If $w\in W$, we say the pair $(\M,w)$ is a {\em pointed model}. A {\em frame of $\mathcal{L}_\vartriangleleft$} is a model of $\mathcal{L}_\vartriangleleft$ without a valuation.

Given a Kripke model $\M=\lr{W,S,V}$ and a world $w\in W$, the semantics of $\mathcal{L}_\vartriangleleft$ is defined as follows.
\[\begin{array}{|lcl|}
\hline
\M,w\vDash p &\iff& w\in V(p)\\
\M,w\vDash \neg A&\iff &\M,w\nvDash A\\
\M,w\vDash A\land B&\iff&\M,w\vDash A\text{ and }\M,w\vDash B\\
\M,w\vDash A\vartriangleleft B&\iff&\text{for all }u,v\in W\text{ such that }S_wuv,\\
&&\text{if }(\M,u\vDash A \iff\M,v\vDash A),\\
&&\text{then }(\M,u\vDash B\iff \M,v\vDash B).\\
\hline
\end{array}
\]

The notions of truth, validity and satisfiability are defined as usual. For instance, formula $A$ is {\em true} at $w$ in $\M$, if $\M,w\vDash A$, in this case we also write $w\vDash A$ when $\M$ is clear; $A$ is {\em valid on a class of frames $F$}, written $F\vDash A$, if for all frames $\mathcal{F}$ in $F$, for all models $\M$ based on $\mathcal{F}$, and for all worlds $w$ in $\mathcal{M}$, we have $\M,w\vDash A$; $A$ is {\em valid}, written $\vDash A$, if for all classes of frames $F$, we have $F\vDash A$; $A$ is {\em satisfiable}, if there is a model $\M$ and a world $w$ in $\M$ such that $\M,w\vDash A$.

\weg{\[\begin{array}{|lcl|}
\hline
\M,w\vDash \vartriangleleft(A_1,\cdots,A_n; B)&\iff&\text{for all }u,v\in\M\text{ such that }S_wuv, \text{ if }(\M,u\vDash A_i \iff\M,v\vDash A_i)\\
&&\text{ for all }i\in\{1,\cdots,n\},\text{ then }(\M,u\vDash B\iff \M,v\vDash B).\\
\hline
\end{array}
\]}

Our operator $\vartriangleleft$ conform to the core idea of supervenience in philosophy: fixing the subvenient fixes the supervenient. Note that our language is an extension of the language of the modal logic of agreement $\mathcal{L}_O$ (see Section~\ref{subsec.agreement}), since one may easily verify that $\vDash OB\lra (\top\vartriangleleft B)$.\footnote{We could equally well use $\bot$ in place of $\top$ here.} In what follows, for simplicity's sake, we will use $OB$ to abbreviate $\top\vartriangleleft B$, and $A\Lleftarrow\Rrightarrow B$ to abbreviate $(A\Rrightarrow B) \land (B\Rrightarrow A)$. Intuitively, $OB$ says that the truth value of $B$ is fixed, and $A\Lleftarrow\Rrightarrow B$ says that $A$ and $B$ supervene on each other.

\weg{The semantics of $\vartriangleleft$ reminds us of that of strict implication $\prec$ defined as:
\[\begin{array}{lcl}
\M,w\vDash A\prec B&\iff \text{for all }u\in \M\text{ such that }Rwu,\text{ if }\M,u\vDash A,\text{ then }\M,u\vDash B.
\end{array}\]
Roughly speaking, just as strict implication is a localized object-language connective corresponding to inference-determined consequence (in the 1-premiss case), so the new connective is a localized object-language connective corresponding to supervenience-determined consequence (in the 1-premiss case).}

The operator $\vartriangleleft$ is not definable/expressible in terms of $\prec$. Take two models $\M$ and $\M'$ as an example, where in $\M$, $S_wuv$ with $p,q$ are only true at $w$, while in $\M'$, $S'_{w'}u'v'$ with $p$ is true at all worlds, and $q$ is true only at $w'$ and $u'$, and accessibility relations are empty on both models. One may check that $(\M,w)$ and $(\M',w')$ can be distinguished by $p\vartriangleleft q$, but cannot be distinguished by any formula of the language which extends propositional logic with the operator $\prec$.

As discussed in~\cite[Subsection~3.34]{Humberstone:connectivebook} (also see~\cite[pp.~188--189]{Humberstone13:logicalrelations} for a summary), every binary connective gives rise to binary relations between formulas in two ways: a local way and a global way. The same story goes with our binary connective/modality $\vartriangleleft$: given a model $\M$ and a world $w$ therein, we can define a (local) relation $R^\vartriangleleft_{\M,w}$ and a (global) relation $R^\vartriangleleft_\M$, respectively, as $\{(A,B)\mid M,w\vDash A\vartriangleleft B\}$ and $\{(A,B)\mid M\vDash A\vartriangleleft B\}$ (or equivalently, $\bigcap_{w\in\M} R^\vartriangleleft_{\M,w}$). Then these two relations are both preorder, that is, reflexive and transitive relations, as easily follows from Prop.~\ref{prop.preorder}, where item $(i)$ says that everything is supervenient on itself, and item $(ii)$ says that if one thing is supervenient on a thing which supervenes on another thing, then the first thing supervenes on the third thing. In short, the operator $\vartriangleleft$ gives rise to two kinds of supervenience relations between formulas.



\begin{proposition}\label{prop.preorder}
Let $A,B,C\in L_\vartriangleleft$.
\begin{enumerate}
\item[(i)]\label{i} $\vDash A\vartriangleleft A$ (Supervenience is reflexive),
\item[(ii)]\label{ii} $\vDash(A\vartriangleleft B)\land(B\vartriangleleft C)\to (A\vartriangleleft C)$ (Supervenience is transitive).
\item[(iii)] $\nvDash(A\vartriangleleft B)\to (B\vartriangleleft A)$ (Supervenience is {\em not} symmetric).
\end{enumerate}
\end{proposition}

Intuitively, if the truth value of $B$ is fixed, then the truth value of $B$ is still fixed, no matter whether the concerned subvenient is fixed.
\begin{fact}\label{fact.O}
$\vDash OB\to A\vartriangleleft B$.
\end{fact}

\begin{proof}
Straightforward by the semantical definitions of $O$ and $\vartriangleleft$.
\end{proof}

If the truth value of $A$ is fixed, then ``$B$ supervenes on $A$'' is amount to that the truth value of $B$ is fixed.
\begin{proposition}\label{prop.equiv}
$\vDash OA\to(A\vartriangleleft B\lra OB)$.
\end{proposition}

\begin{proof}
Let $\M=\lr{W,S,V}$ be a model and $w\in W$ such that $\M,w\vDash OA$. By Fact~\ref{fact.O}, we have that $\M,w\vDash OB\to A\vartriangleleft B$. Left to show is the other direction. To this end, assume, for a contradiction, that $\M,w\vDash A\vartriangleleft B$ and $\M,w\nvDash OB$. Then there exist $u,v\in W$ such that $S_wuv$, and it is {\em not} the case that $(\M,u\vDash B\iff \M,v\vDash B)$. By the fact that $\M,w\vDash OA$ and $S_wuv$, we obtain that $(\M,u\vDash A\iff \M,v\vDash A)$. Thus we have found two worlds $u,v\in W$ with $S_wuv$ and $(\M,u\vDash A\iff \M,v\vDash A)$ but {\em not} $(\M,u\vDash B\iff \M,v\vDash B)$, which is contrary to the assumption.
\end{proof}

\section{Comparing the expressive powers of $\mathcal{L}_\vartriangleleft$ and $L_O$}\label{sec.exp-ls-lo}

In this section, we compare the relative expressive powers of $\mathcal{L}_\vartriangleleft$ and $\mathcal{L}_O$. The prime result is that, $\mathcal{L}_\vartriangleleft$ is more expressive than $\mathcal{L}_O$. We first introduce the definition of expressivity.
\begin{definition}[Expressivity]
Let two logical languages $L1$ and $L2$ be interpreted in the same class $M$ of models.
\begin{itemize}
\item $L2$ is {\em at least as expressive as} $L1$, if for every formula $A$ in $L1$, there is a formula $B$ in $L2$ such that for all $\M$ in $M$ and $w$ in $\M$, we have $\M,w\vDash A$ iff $\M,w\vDash B$.
\item $L2$ is {\em more expressive than} $L1$, or $L1$ is less expressive than $L2$, if $L2$ is at least as expressive as $L1$ but not vice versa.
\item $L2$ is {\em equally expressive as} $L1$, if $L2$ is at least as expressive as $L1$ and vice versa.
\end{itemize}
\end{definition}

As noted above, $OB$ is definable in terms of $\vartriangleleft$, as $OB=_{df}\top\vartriangleleft B$, thus $\mathcal{L}_\vartriangleleft$ is at least as expressive as $\mathcal{L}_O$. To show that $\mathcal{L}_\vartriangleleft$ is more expressive than $\mathcal{L}_O$, we only need to show that $\mathcal{L}_O$ is {\em not} at least as expressive as $L_\vartriangleleft$. Observe that even the simple formula $p\vartriangleleft q$ in $\mathcal{L}_\vartriangleleft$ seems to be not definable with $\mathcal{L}_O$ formulas. Thus we only need to construct two models that are distinguishable by $p\vartriangleleft q$ but not by any $\mathcal{L}_O$ formulas. Before doing this, let us illustrate the non-triviality of the construction with some discussions.

Let two pointed models $(\M,w)$ and $(\M',w')$ be given. Firstly, consider the case where $Op$ holds at both $w$ and $w'$. In this case, from Prop.~\ref{prop.equiv}, it follows that $p\vartriangleleft q\lra Oq$ are both true at $w$ and $w'$. Thus although $Op$ cannot distinguish $(\M,w)$ and $(\M',w')$, $p\vartriangleleft q$ cannot distinguish either, because otherwise $Oq$ can distinguish the two pointed models too. This is not consistent with our goal, and hence $Op$ should be false at both $w$ and $w'$. Secondly, consider the case where $Oq$ holds at both $w$ and $w'$. In this case, by Fact~\ref{fact.O}, we have that $p\vartriangleleft q$ holds at $w$ and $w'$ as well, thus $p\vartriangleleft q$ cannot distinguish the two pointed models. This is contrary to our goal too, and hence $Oq$ should be also false at both $w$ and $w'$. All in all, to construct two models that are distinguishable by $p\vartriangleleft q$ but not by any $\mathcal{L}_O$ formulas, we need to construct two models where $Op$ and $Oq$ are both false at the designated worlds.

Before constructing the desired models, we define a notion of bisimulation for $\mathcal{L}_O$, which we call `$O$-bisimulation'.\footnote{The usage `$O$-bisimulation' may be a bit loose, since we are unsure whether the Hennessy-Milner-style theorem holds for it. But this does not affect the results below.}

\begin{definition}[$O$-Bisimulation] Let $\M=\lr{W,S,V}$ and $\M'=\lr{W',S',V'}$ be models. Say $\mathcal{R}\subseteq W\times W'$ is an {\em $O$-bisimulation} between $\M$ and $\M'$, if $\mathcal{R}$ is nonempty, and if $w\mathcal{R}w'$, then the following conditions are satisfied:
\begin{enumerate}
\item[(Atom)] $w$ and $w'$ satisfy the same propositional variables.
\item[($O$-Zig)] if for all $u,v,x\in W$ such that $S_wuv$ and $S_wux$, then there exist $u',v',x'\in W'$ such that $S'_{w'}u'v'$ and $S'_{w'}u'x'$, and there are $y',z_1',z_2'\in\{u',v',x'\}$ such that $u\mathcal{R}y'$, $v\mathcal{R}z'_1$ and $x\mathcal{R}z_2'$.
\item[($O$-Zag)] if for all $u',v',x'\in W'$ such that $S_{w'}u'v'$ and $S_{w'}u'x'$, then there exist $u,v,x\in W$ such that $S_{w}uv$ and $S_{w}ux$, and there are $y,z_1,z_2\in\{u,v,x\}$ such that $y\mathcal{R}u'$, $z_1\mathcal{R}v'$ and $z_2\mathcal{R}x'$.
\end{enumerate}
We say that $(\M,w)$ and $(\M',w')$ are {\em $O$-bisimilar}, written $(\M,w)\bis_O(\M',w')$, if there is an $O$-bisimulation between $\M$ and $\M'$ that contains $(w,w')$.
\end{definition}

It is instructive to give some explanations for the conditions ($O$-Zig) and ($O$-Zag). Intuitively, ($O$-Zig) says that if $w\mathcal{R}w'$ and $u$ has two (not necessarily different) successors $v$ and $x$ with respect to the accessibility relation $S_w$, then there exists $u'$ in $\M'$ such that $u'$ has two (not necessarily different) successors $v'$ and $x'$ with respect to the accessibility relation $S'_{w'}$, and each of $u,v,x$ is $\mathcal{R}$-related to {\em at least one of} $u',v',x'$ (possibly not in order). The intuitive meaning of ($O$-Zag) is similar.

It is easy to see that $\bis_O$ is the largest $O$-bisimulation and an equivalence relation. The proposition below states that, $\mathcal{L}_O$ is not able to tell apart any two $O$-bisimilar pointed models. That is, any formula in $\mathcal{L}_O$ is invariant under $O$-bisimulation.

\begin{proposition}\label{prop.invariance}
Where $\M=\lr{W,S,V}$ and $\M'=\lr{W',S',V'}$ are models such that $w\in W$ and $w'\in W'$, if $(\M,w)\bis_O(\M',w')$, then $(\M,w)\equiv_O(\M',w')$.
\end{proposition}

\begin{proof}
Suppose that $(\M,w)\bis_O(\M',w')$. We show by induction that  for all $A\in \mathcal{L}_O$, we have $$\M,w\vDash A\iff \M',w'\vDash A.$$
The Boolean cases are trivial. We need only show the case for $OA$.

Assume that $\M,w\nvDash OA$. Then there are $u,v\in W$ such that $S_wuv$ and $\M,u\vDash A$ and $\M,v\nvDash A$. We consider two cases.
\begin{itemize}
\item There is no $x$ such that $x\neq v$ and $S_wux$. By (O-Zig), there are $u',v',x'\in W'$ such that $S'_{w'}u'v'$ and $S'_{w'}u'x'$, and there are  $y',z_1',z_2'\in\{u',v',x'\}$ such that $u\bis_Oy'$ and $v\bis_Oz_1'$ and $v\bis_Oz_2'$. We now consider the following subcases.
    \begin{itemize}
    \item $y'=u'$. Since $u\bis_Oy'$, by induction hypothesis and $\M,u\vDash A$, we have $\M',u'\vDash A$. Similarly, from $v\bis_Oz_1'$ it follows that $z_1'\nvDash A$. Then $z_1'\neq u'$, and thus $z_1'=v'$ or $z_1'=x'$. Hence $v'\nvDash A$ or $x'\nvDash A$, either of which entails $\M',w'\nvDash OA$.
    \item $y'=v'$. Since $u\bis_Oy'$, by induction hypothesis and $\M,u\vDash A$, we have $\M',v'\vDash A$. By a similar argument to the first subcase, from $v\bis_Oz_1'$ we can obtain that $z_1'\neq v'$, which means that $z_1'=u'$ or $z_1'=x'$. If $z_1'=u'$, then $u'\nvDash A$, and therefore $\M',w'\nvDash OA$. If $z_1'=x'$, then $x'\nvDash A$, due to $S'_{w'}u'v'$ and $S'_{w'}u'x'$ and $v'\vDash A$, no matter whether $u'\vDash A$ or $u'\nvDash A$, we both have $\M',w'\nvDash OA$.
    \item $y'=x'$. Analogous to the second subcase, we get that $\M',w'\nvDash OA$.
    \end{itemize}
    Either subcase implies that $\M',w'\nvDash OA$.
    \weg{\begin{itemize}
    \item $y'=u'$, $z_1'=v'$ and $z_2'=x'$. Since $u\bis_Oy'$, by induction hypothesis we have $u'\vDash A$; similarly, as $v\bis_Oz_1'$, we infer $v'\nvDash A$. Thus $\M',w'\nvDash OA$.
    \item $y'=u'$, $z_1'=x'$ and $z_2'=v'$. Similar to the first subcase, we can obtain that $\M',w'\nvDash OA$.
    \item $y'=v'$, $z_1'=u'$ and $z_2'=x'$. Since $u\bis_Oy'$, by induction hypothesis we have $v'\vDash A$; similarly, as $v\bis_Oz_1'$, we infer $u'\nvDash A$. Thus $\M',w'\nvDash OA$.
    \item $y'=v'$, $z_1'=x'$ and $z_2'=u'$. Similar to the third subcase, we can derive that $\M',w'\nvDash OA$.
    \item $y'=x'$, $z_1'=u'$ and $z_2'=v'$. Since $u\bis_Oy'$, by induction hypothesis we have $x'\vDash A$; similarly, as $v\bis_Oz_1'$, we infer $u'\nvDash A$. Thus $\M',w'\nvDash OA$.
    \item $y'=x'$, $z_1'=v'$ and $z_2'=u'$. Similar to the fifth subcase, we can derive that $\M',w'\nvDash OA$.
    \end{itemize}
    Either subcase implies that $\M',w'\nvDash OA$.}
\item There is an $x$ such that $x\neq v$ and $S_wux$. By (O-Zig), there are $u',v',x'\in W'$ such that $S'_{w'}u'v'$ and $S'_{w'}u'x'$, and there are  $y',z_1',z_2'\in\{u',v',x'\}$ such that $u\bis_Oy', v\bis_Oz_1'$ and $x\bis_Oz_2'$. Analogous to the first case, we need to check three subcases. \begin{itemize}
    \item $y'=u'$. Since $u\bis_Oy'$, by induction hypothesis and $\M,u\vDash A$, we have $\M',u'\vDash A$. Since $v\bis_Oz_1'$ and $\M,v\nvDash A$, using induction hypothesis, we infer $\M',z_1'\nvDash A$. Thus $z_1'\neq u'$, i.e. $z_1'=v'$ or $z_1'=x'$, then $v'\nvDash A$ or $x'\nvDash A$, either of which implies $\M',w'\nvDash OA$.
    \item $y'=v'$. Since $u\bis_Oy'$, by induction hypothesis and $\M,u\vDash A$, we have $\M',v'\vDash A$. By a similar argument to the first subcase, from $v\bis_Oz_1'$ we can obtain that $z_1'\neq v'$, which means that $z_1'=u'$ or $z_1'=x'$. If $z_1'=u'$, then $u'\nvDash A$, and therefore $\M',w'\nvDash OA$. If $z_1'=x'$, then $x'\nvDash A$, due to $S'_{w'}u'v'$ and $S'_{w'}u'x'$ and $v'\vDash A$, no matter whether $u'\vDash A$ or $u'\nvDash A$, we both have $\M',w'\nvDash OA$.
    \item $y'=x'$. Analogous to the second subcase, we get that $\M',w'\nvDash OA$.
    \end{itemize}
      Again, either subcase implies that $\M',w'\nvDash OA$.
 \weg{We consider only two subcases, others are left as an exercise.
    \begin{itemize}
    \item $y'=v'$, $z_1'=x'$ and $z_2'=u'$. Since $u\bis_Oy'$, by induction hypothesis we have $v'\vDash A$; similarly, as $v\bis_Oz_1'$, we infer that $x'\nvDash A$. If $u'\vDash A$, from $x'\nvDash A$ it follows that $\M',w'\nvDash OA$; if $u'\nvDash A$, from $v'\vDash A$ we also obtain $\M',w'\nvDash OA$.
     \item $y'=x'$, $z_1'=v'$ and $z_2'=u'$. Since $u\bis_Oy'$, by induction hypothesis we have $x'\vDash A$; similarly, as $v\bis_Oz_1'$, we derive $v'\nvDash A$. If $u'\vDash A$, from $v'\nvDash A$ it follows that $\M',w'\nvDash OA$; if $u'\nvDash A$, from $x'\vDash A$ we also obtain $\M',w'\nvDash OA$.
    \end{itemize}
      Again, either subcase implies that $\M',w'\nvDash OA$.}
\end{itemize}
In either case, we conclude that $\M',w'\nvDash OA$. The converse is similar.
\end{proof}

\begin{proposition}\label{prop.moreexp}
$\mathcal{L}_\vartriangleleft$ is more expressive than $\mathcal{L}_O$.
\end{proposition}

\begin{proof}
Consider the models $\M=\lr{W,S,V}$ and $\M'=\lr{W',S',V'}$, where
\begin{itemize}
\item $W=\{s,t,u,v\}$, $S_s=\{(t,u),(t,v)\}$, $S_t=S_u=S_v=\emptyset$, $V(p)=\{s,t\}$, $V(q)=\{s,v\}$;
\item $W'=\{s',t',u',v'\}$, $S'_{s'}=\{(t',u'),(t',v')\}$, $S'_{t'}=S'_{u'}=S'_{v'}=\emptyset$, $V'(p)=\{s',v'\}$, $V'(q)=\{s',t'\}$.
\end{itemize}

The two models are visualized below:
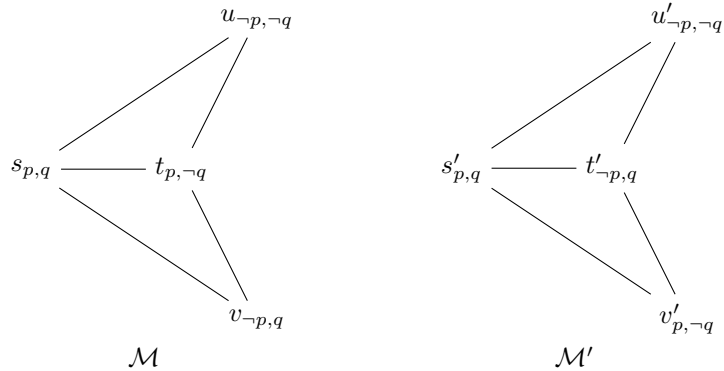
\begin{center}
\begin{tikzpicture}
\node (0) at (0,0) {$s_{p,q}$};
\node (1) at (2,0) {$t_{p,\neg q}$};
\node (2) at (3,2){$u_{\neg p,\neg q}$};
\node (3) at (3,-2){$v_{\neg p,q}$};
\node (4) at (1.5,-2.5) {$\M$};
\draw[-] (0) to (1);
\draw[-] (0) to (2);
\draw[-] (0) to (3);
\draw[-] (1) to (2);
\draw[-] (1) to (3);
\end{tikzpicture}
\hspace{1.5cm}
\begin{tikzpicture}
\node (0) at (0,0) {$s'_{p,q}$};
\node (1) at (2,0) {$t'_{\neg p, q}$};
\node (2) at (3,2){$u'_{\neg p,\neg q}$};
\node (3) at (3,-2){$v'_{p,\neg q}$};
\node (4) at (1.5,-2.5) {$\M'$};
\draw[-] (0) to (1);
\draw[-] (0) to (2);
\draw[-] (0) to (3);
\draw[-] (1) to (2);
\draw[-] (1) to (3);
\end{tikzpicture}
\end{center}
Define $\mathcal{R}=\{(s,s'),(t,v'),(v,t'),(u,u')\}$. One may easily verify that $\mathcal{R}$ is indeed a $O$-bisimulation between $\M$ and $\M'$. Since $(s,s')\in\mathcal{R}$, we have $(\M,s)\bis_O(\M',s')$. Due to Prop.~\ref{prop.invariance}, $(\M,s)$ and $(\M',s')$ cannot be distinguished by any $\mathcal{L}_O$ formulas.

However, the two pointed models can be distinguished by an $\mathcal{L}_\vartriangleleft$ formula $p\vartriangleleft q$, since $\M,s\vDash p\vartriangleleft q$ but $\M',s'\nvDash p\vartriangleleft q$. To see $\M',s'\nvDash p\vartriangleleft q$, just notice that $S'_{s'}t'u'$ and $(t'\vDash p\iff u'\vDash p)$ but it is {\em not} the case that $(t'\vDash q\iff u'\vDash q)$.
\end{proof}

\section{A sound proof system}\label{sec.soundsystem}

In this section we present a proof system for $\mathcal{L}_\vartriangleleft$ and show its soundness with respect to the class of all frames.

\begin{definition}[Proof system $\mathbb{LS}$] The proof system $\mathbb{LS}$ consists of the following axiom schemas and inference rules.
\[ \begin{array}{ll}
\text{A0} & \text{all instances of tautologies}  \\

\text{A1} &  OB\to (A\vartriangleleft B)\\

\text{A2} & (A\vartriangleleft B)\to(OA\to OB)\\

\text{A3} & O(A\lra B)\lra (A\Lleftarrow\Rrightarrow B)\\

\text{A4} & (A\vartriangleleft B)\land(B\vartriangleleft C)\to (A\vartriangleleft C)\\

\text{A5} & (A\vartriangleleft B_1)\land\cdots\land (A\vartriangleleft B_n)\to (A\vartriangleleft\sharp(B_1,\cdots,B_n)),\\
&\text{where }\sharp\text{ is an }n\text{-ary Boolean connective.}\\

\text{R1} & \text{from }A\to B\text{ and }A\text{ infer }B\\

\text{R2} & \text{from }A \text{ infer }OA\\

\end{array} \]
\end{definition}

The intuition of the axioms can be explained as below. \text{A1} says if the truth value of a formula is fixed, then no matter whether the truth value of the subvenient is fixed, the truth value of the formula is still fixed; \text{A2} says fixing the truth value of the subvenient fixes the truth value of the supervenient, which characterizes the core idea of supervenience; \text{A3} says that, saying the truth value of a biconditional is fixed, amounts to saying that its conditionals supervene on/determine each other; \text{A4} says that supervenience is transitive (see the preceding paragraph of Prop.~\ref{prop.preorder}); \weg{\text{A5} corresponds to `compositionality principle': the truth value of a Boolean formula supervenes on/is determined by its components; in other words, if the truth value of its components are fixed, the truth value of the Boolean formula itself is also fixed.} \text{A5} can be seen as a counterpart in $\mathcal{L}_\Rrightarrow$ of the validity of the supervenience-determined consequence relation between $\{B_1,\cdots,B_n\}$ and $\sharp(B_1,\cdots,B_n)$, i.e. $B_1,\cdots,B_n\Vdash_\mathcal{V}\sharp(B_1,\cdots,B_n)$, which means that once the truth value of each $B_i$ is fixed, the truth value of their Boolean compound is also fixed.

\begin{proposition}\
\begin{enumerate}
\item\label{item11} $\vdash A\vartriangleleft A$
\item\label{item1} $\vdash A\Lleftarrow\Rrightarrow\neg A$
\item\label{item2} $\vdash OA\to((A\vartriangleleft B)\lra OB)$.
\item\label{item3} $\vdash O(A\lra B)\to (OA\lra OB)$.
\item\label{item4} $\vdash A\vartriangleleft B_1\land\cdots\land A\vartriangleleft B_m\to A\vartriangleleft B$, where $B$ is a Boolean compound of $B_1,\cdots,B_m$.
\item\label{item5} $\vdash (OA_1\land \cdots\land OA_n)\to O\sharp(A_1,\cdots,A_n)$, where $\sharp$ is an $n$-ary Boolean connective.
\item\label{item7} $\vdash (A\Lleftarrow\Rrightarrow B)\to (C\vartriangleleft A\lra C\vartriangleleft B)\land (A\vartriangleleft C\lra B\vartriangleleft C)$
\item\label{item8} If $\vdash A$, then $\vdash B\vartriangleleft A.$
\item\label{item9} If $\vdash A\lra B$, then $\vdash A\vartriangleleft B$.
\item\label{item10} If $\vdash A\lra B$, then $\vdash OA\to OB$.
\end{enumerate}
\end{proposition}
Item~\ref{item11} concerns the reflexivity of supervenience: everything supervenes on itself (see the preceding paragraph of Prop.~\ref{prop.preorder}). Item~\ref{item1} can be understood in a way that every formula and its negation are entirely about the same subject-matter~\cite{Humberstone:1993}. The intuition of item~\ref{item2} can be seen from the discussion before Prop.~\ref{prop.equiv}. Item~\ref{item3} says that if the truth value of a biconditional is fixed, then the truth value of one of its sides is fixed if and only if the truth value of another is fixed. Item~\ref{item4} corresponds to `compositionality principle': the truth value of a Boolean formula supervenes on/is determined by its components; in other words, if the truth value of its components are fixed, the truth value of the Boolean formula itself is also fixed. Item~\ref{item7} says that if two formulas supervene on each other, then either of them can be replaced with the other, no matter whether they are the subvenient or the supervenient of other formulas. Item~\ref{item8} says that provable formulas supervene on anything. Item~\ref{item9} says provable equivalents supervene on each other. Note that items~\ref{item5} and~\ref{item10} are respectively the axiom schema (OComp) and the inference rule (OCong) of the proof system $\mathbf{LO}$ in~\cite{Humberstone:2002}.

\begin{proposition}[Soundness of $\mathbb{LS}$] The proof system $\mathbb{LS}$ is sound with respect to the class of all frames.
\end{proposition}

\begin{proof}
We only show the validity of axiom \text{A3}. Let a pointed model $(\M,w)$ where $\M=\lr{W,S,V}$ be given.

Firstly, suppose $\M,w\vDash O(A\lra B)$, we need to show that $\M,w\vDash A\Lleftarrow\Rrightarrow B$, that is to show, $\M,w\vDash (A\Rrightarrow B)\land (B\Rrightarrow A)$. We show $\M,w\vDash A\Rrightarrow B$, the proof for $\M,w\vDash B\Rrightarrow A$ is similar. Assume for any $u,v\in W$ such that $S_wuv$ and $(\M,u\vDash A\iff \M,v\vDash A)$. From the supposition and $S_wuv$, it follows that $(\M,u\vDash A\lra B\iff \M,v\vDash A\lra B)$. Then it is easy to show that $(\M,u\vDash B\iff \M,v\vDash B)$. Therefore, $\M,w\vDash A\Rrightarrow B$.

Conversely, suppose $\M,w\vDash A\Lleftarrow\Rrightarrow B$, i.e., $\M,w\vDash (A\Rrightarrow B)\land (B\Rrightarrow A)$. Assume for any $u,v\in W$ such that $S_wuv$, we need to show that $(\M,u\vDash A\lra B\iff \M,v\vDash A\lra B)$. By supposition and $S_wuv$, we obtain that $(\M,u\vDash A\iff \M,v\vDash A)\iff (\M,u\vDash B\iff \M,v\vDash B)$. From this it follows that $(\M,u\vDash A\lra B\iff \M,v\vDash A\lra B)$, as required.
\end{proof}

\section{Generalized supervenience logics}\label{sec.generalizedsup}

So far we have been devoted to the dyadic supervenience operator, a function taking a pair of formulas as arguments. We now generalize this case into the case when the supervenience operator takes any finitely many formulas as arguments. In detail,  the generalized language (denoted $\mathcal{L}_{\vartriangleleft^\infty}$) is defined as
$$A::= p\mid \neg A\mid A\land A\mid (A,\cdots,A)\vartriangleleft A$$
Where the construct $(A,\cdots,A)\vartriangleleft A$ contains $n+1$ formulas for any $n\in\mathbb{N}$.

The Kripke model of $\mathcal{L}_{\vartriangleleft^\infty}$ is defined as that of $\mathcal{L}_\vartriangleleft$. The new construct $(A_1,\cdots,A_n)\vartriangleleft B$ is interpreted as follows:
\[\begin{array}{|lcl|}
\hline
\M,w\vDash (A_1,\cdots,A_n)\vartriangleleft B&\iff&\text{for all }u,v\in\M\text{ such that }S_wuv, \\
&&\text{if }(\M,u\vDash A_i \iff\M,v\vDash A_i)\text{ for all }i\in\{1,\cdots,n\},\\
&&\text{then }(\M,u\vDash B\iff \M,v\vDash B).\\
\hline
\end{array}
\]
Recall the semantics of the operator $D$  in Sec.~\ref{sec.determinacy}. One may easily see that the interpretation of the operator $D$ in the language $\mathcal{L}_D$ is a special case of that of $\vartriangleleft$ in $\mathcal{L}_{\vartriangleleft^\infty}$, when $S$ is defined in a way such that $S_wuv$ just in case $wRu$ and $wRv$ for all $w,u,v$ in the underlying model.

We could also consider a class of languages $\mathcal{L}_{\vartriangleleft^n}$ for all $n\in \mathbb{N}$, defined inductively as follows:
$$A::= p\mid \neg A\mid A\land A\mid (A,\cdots,A)\vartriangleleft^nA$$
Where $\vartriangleleft^n$ is an $n+1$-ary operator for each $n$ and interpreted by the following:
\[\begin{array}{|lcl|}
\hline
\M,w\vDash (A_1,\cdots,A_n)\vartriangleleft^n B&\iff&\text{for all }u,v\in\M\text{ such that }S_wuv, \\
&&\text{if }(\M,u\vDash A_i \iff\M,v\vDash A_i)\text{ for all }i\in\{1,\cdots,n\},\\
&&\text{then }(\M,u\vDash B\iff \M,v\vDash B).\\
\hline
\end{array}
\]
Then the agreement operator $O$ in Section~\ref{subsec.agreement} is $\vartriangleleft^0$, and the supervenience operator $\vartriangleleft$ in Section~\ref{sec.supervenience} is $\vartriangleleft^1$. Accordingly, $\mathcal{L}_O$ is $\mathcal{L}_{\vartriangleleft^0}$ and $\mathcal{L}_\vartriangleleft$ is $\mathcal{L}_{\vartriangleleft^1}$.

It should be clear that for all $n\in\mathbb{N}$, $(A_1,\cdots,A_n)\vartriangleleft^n B$ is logically equivalent to $(A_1,\cdots,A_n)\vartriangleleft B$. This implies that $\mathcal{L}_{\vartriangleleft^\infty}$ is an extension of $\mathcal{L}_{\vartriangleleft^n}$ for any $n\in \mathbb{N}$, and thus $\mathcal{L}_{\vartriangleleft^\infty}$ is at least as expressive as $\mathcal{L}_{\vartriangleleft^n}$ for any $n\in \mathbb{N}$.

Moreover, we have seen that $\vartriangleleft^0$ is definable in terms of $\vartriangleleft^1$, as $OA=_{df} \top\vartriangleleft^1 A$, or equivalently, $\bot\vartriangleleft^1 A$. In general, $\vartriangleleft^n$ is definable in terms of $\vartriangleleft^{n+1}$, as $(A_1,\cdots,A_n)\vartriangleleft^n B=_{df} (\top,A_1,\cdots,A_n)\vartriangleleft^{n+1} B$, or equivalently, $(\bot,A_1,\cdots,A_n)\vartriangleleft^{n+1} B$.

Similar to Fact~\ref{fact.O} and Prop.~\ref{prop.equiv}, we can show
\begin{proposition}For all $A,B,C\in \mathcal{L}_{\Rrightarrow^2}$,
\begin{enumerate}
\item $\vDash(C\Rrightarrow^1 B)\to ((C,A)\Rrightarrow^2 B).$
\item $\vDash(C\Rrightarrow^1A)\to (((C,A)\Rrightarrow^2 B)\lra (C\Rrightarrow^1 B))$.
\end{enumerate}
\end{proposition}

In general, we have
\begin{proposition} Let $n\in\mathbb{N}$. For all $A_1,\cdots,A_n,A_{n+1},B$,
\begin{enumerate}
\item $\vDash((A_1,\cdots,A_{n})\Rrightarrow^n B)\to ((A_1,A_2,\cdots,A_{n+1})\Rrightarrow^{n+1} B).$
\item $\vDash((A_1,\cdots,A_{n})\Rrightarrow^n A_{n+1})\to (((A_1,\cdots,A_{n+1})\Rrightarrow^{n+1} B)\lra ((A_1,\cdots,A_{n})\Rrightarrow^n B))$.
\end{enumerate}
\end{proposition}

In Prop.~\ref{prop.moreexp}, we have shown that $\mathcal{L}_{\vartriangleleft^1}$ is more expressive than $\mathcal{L}_{\vartriangleleft^0}$. We guess the result can be generalized to the following, which we leave for future work.

\begin{conjecture}\label{conj.exp-ls}
For all $n\in\mathbb{N}$, $\mathcal{L}_{\vartriangleleft^{n+1}}$ is more expressive than $\mathcal{L}_{\vartriangleleft^n}$.
\end{conjecture}

\weg{\[\begin{array}{|lcl|}
\hline
\M,w\vDash (A_1,\cdots,A_n)\vartriangleleft^n_m (B_1,\cdots,B_m)&\iff&\text{for all }u,v\in\M\text{ such that }S_wuv, \\
&&\text{if }(\M,u\vDash A_i \iff\M,v\vDash A_i)\text{ for all }i\in\{1,\cdots,n\},\\
&&\text{then }(\M,u\vDash B_j\iff \M,v\vDash B_j)\text{ for all }j\in\{1,\cdots,m\}.\\
\hline
\end{array}
\]}

In what follows, for the sake of presentation, we define $u^\M(A)$ for all formulas $A$ and all valuations $u$ such that $u^\M(A)=T$ iff $\M,u\vDash A$, and we will drop the superscript $\M$ when it is clear.
\weg{\[
u^\M(A)=\left\{
  \begin{array}{ll}
    1, & \hbox{if $\M,u\vDash A$;} \\
    0, & \hbox{otherwise.}
  \end{array}
\right.
\]}

So far we have considered the supervenience of a formula on another (or a set of formulas). We can generalize this kind of supervenience to the supervenience of a set of formulas on another set.
For all $n,m\in\mathbb{N}$,
\[\begin{array}{|lcl|}
\hline
\M,w\vDash (A_1,\cdots,A_n)\vartriangleleft^n_m (B_1,\cdots,B_m)&\iff&\text{for all }u,v\in\M\text{ such that }S_wuv, \\
&&\text{if }u(A_i)=v(A_i)\text{ for all }i\in\{1,\cdots,n\},\\
&&\text{then }u(B_j)=v(B_j)\text{ for all }j\in\{1,\cdots,m\}.\\
\hline
\end{array}
\]

Intuitively, it means that the set of formulas $B=\{B_1,\cdots,B_m\}$ supervenes on another set $A=\{A_1,\cdots,A_n\}$. One may easily verify that the empty set supervenes on everything.

Note that $\vartriangleleft^n_m$ is interdefinable with $\vartriangleleft^n$ (viz. $\vartriangleleft^n_1$). Firstly, the operator $\vartriangleleft^n$ is definable with $\vartriangleleft^n_m$, as
$$((A_1,\cdots,A_n)\vartriangleleft^nB)=_{df}((A_1,\cdots,A_n)\vartriangleleft^n_m (B,\cdots,B)),$$
where $B$ occurs $m$ times. Conversely, $\vartriangleleft^n_m$ is definable in terms of $\vartriangleleft^n$, as
$$((A_1,\cdots,A_n)\vartriangleleft^n_m(B_1,\cdots,B_m))=_{df} \bigwedge_{1\leq i\leq m}((A_1,\cdots,A_n)\vartriangleleft^n B_i).$$
This is similar to the interdefinability result of the relations `supervenience$_1$' (between a property and a class of properties) and `supervenience' (between two classes of properties)~\cite[pp.~102-103]{Humberstone:1992}.

\medskip

\weg{
In what follows, we will obtain a much general result, i.e., for all $n\in\mathbb{N}$, $\mathcal{L}_{\vartriangleleft^{n+1}}$ is more expressive than $\mathcal{L}_{\vartriangleleft^n}$. We use the same strategy as in Section~\ref{sec.exp-ls-lo}.

First, we define a notion of bisimulation for $\mathcal{L}_{\vartriangleleft^n}$.

\begin{proposition}
For all $n\in\mathbb{N}$, $\mathcal{L}_{\vartriangleleft^{n+1}}$ is more expressive than $\mathcal{L}_{\vartriangleleft^n}$.
\end{proposition}}

Recall the determinacy operator $D$ in Section~\ref{sec.determinacy}, whose semantics is a special case of that of the generalized supervenience operator $\vartriangleleft$. Like $\vartriangleleft$, We can do the similar thing for the operator $D$, and denote $\mathcal{L}_{D^0}$, $\mathcal{L}_{D^1}$, $\cdots$, $\mathcal{L}_{D^{n}}$, respectively, the 1-argument, 2-argument, $\cdots$, $n+1$-argument fragments of $\mathcal{L}_D$, all of which have, respectively, $D^0(=\Delta)$, $D^1$, $\cdots$, $D^n$ as the sole primitive modalities. Also, we write $\mathcal{L}_{D^\infty}$ for $\mathcal{L}_D$. We can see that $\mathcal{L}_{D^0}$ is $\mathcal{L}_\Delta$. We will show that all $\mathcal{L}_{D^i}$ (where $i\in\mathbb{N}\cup\{\infty\}$) are equally expressive.

As observed in~\cite{Gorankoetal:2016}, when $D$ is defined on arbitrary models, i.e. the accessibility relation $R$ is arbitrary:
\[\begin{array}{lll}
\M,w\vDash D(A_1,\cdots,A_n;B)&\iff&\text{for all }u,v\in W\text{ such that }wRu\text{ and }wRv,\\
&& \text{if }(\M,u\vDash A_i\iff \M,v\vDash A_i)\text{ holds for all }i\leq n,\\
&&\text{then }(\M,u\vDash B\iff \M,v\vDash B).\\
\end{array}\]

Then $D$ is definable in terms of the necessity operator $\Box$. And since $\Box$ is definable in terms of $D$ on the class of reflexive models (but not in general). Thus $\mathcal{L}_D$ is equally expressive as standard modal logic $\mathcal{L}_\Box$ on the class of reflexive models. We will show, however, in general, i.e. on the class of all models, $\mathcal{L}_D$ is less expressive than $\mathcal{L}_\Box$.

In~\cite[Section 8.2]{Gorankoetal:2016}, the authors listed an open research direction on how to axiomatize $\mathcal{L}_D$ over various classes of frames. In this section, we resolve this issue.

\subsection{Comparing the expressive power of $\mathcal{L}_\Delta$ and $\mathcal{L}_D$}\label{sec.exp-lcld}

In this part, we will demonstrate that all $\mathcal{L}_{D^i}$ (where $i\in\mathbb{N}\cup\{\infty\}$) are equally expressive on the class of all models. Specifically, $\mathcal{L}_\Delta$ is equally expressive as $\mathcal{L}_D$ on that class. In the sequel, the accessibility relation $R$ has no any constraints.

It should be clear that $\Delta=(D^0)$ is definable in terms of $D^1$, as $\vDash\Delta B\lra D^1(\top;B)$. Thus $\mathcal{L}_{D^1}$ is an extension of $\mathcal{L}_\Delta$. Also, $\vDash D^1(A;B)\lra D^2(\top,A;B)$, thus $\mathcal{L}_{D^2}$ is an extension of $\mathcal{L}_{D^1}$. In general, we have $\vDash D^n(A_1,\cdots, A_n;B)\lra D^{n+1}(\top,A_1,\cdots,A_n;B)$. Thus for $m\leq n\in\mathbb{N}$, $\mathcal{L}_{D^n}$ is at least as expressive as $\mathcal{L}_{D^m}$, and $\mathcal{L}_{D}(=\mathcal{L}_{D^\infty})$ is at least as expressive as all $\mathcal{L}_{D^i}$. \weg{We will abbreviate $D(\epsilon;B)$ as $\Delta B$. }We will show that $D$ is also definable in terms of $\Delta$.
Before that, we first give some appetizers. We have that
$$D(A;B)\lra(\Delta(A\to B)\vee\Delta(A\to\neg B))\land(\Delta(\neg A\to B)\vee\Delta(\neg A\to\neg B))$$
and
\[\begin{array}{lll}
D(A_1,A_2;B)&\lra&(\Delta(A_1\land A_2\to B)\vee\Delta(A_1\land A_2\to\neg B))\land\\
&&(\Delta(\neg A_1\land A_2\to B)\vee\Delta(\neg A_1\land A_2\to\neg B))\land\\
&&(\Delta(A_1\land \neg A_2\to B)\vee\Delta(A_1\land\neg A_2\to\neg B))\land\\
&&(\Delta(\neg A_1\land\neg A_2\to B)\vee\Delta(\neg A_1\land\neg A_2\to\neg B))\\
\end{array}\]

Now we lift the results to a general level. Let $\{A_1,\cdots,A_n\}$ be a finite nonempty set of formulas. For each $T\subseteq \{1,\cdots,n\}$, let $B_T$ be the conjunction $B_1\land \cdots\land B_n$ such that if $i\in T$, then $B_i=A_i$; otherwise, $B_i=\neg A_i$. The definition of $B_T$ is very similar to Kim's notion of $B$-maximal properties~\cite[p.~58]{Kim:1984}, which though was defined for {\em properties} rather than formulas. The notion was also introduced in~\cite[p.~5]{Gorankoetal:2016}, for different purposes. The $B_T$ is very important in proving our results below, and we thus give some explanations.

Intuitively, the conjuncts of $B_T$ consist of either $A_j$ or $\neg A_j$ for each $j\in\{1,\cdots,n\}$. For example, $B_\emptyset=\neg A_1\land \cdots\land\neg A_n$, $B_{\{1\}}=A_1\land\neg A_2\land\cdots\land \neg A_n$, and $B_{\{1,\cdots,n\}}=A_1\land \cdots\land A_n$. It is easy to check that $\bigvee_{T\subseteq \{1,\cdots,n\}}B_T$ is a tautology.

We are now ready to show the general validity: for all $n\in\mathbb{N}$, we have
\[\begin{array}{lll}
\vDash D(A_1,\cdots,A_n;B)&\lra&\bigwedge_{T\subseteq \{1,\cdots,n\}}(\Delta(B_T\to B)\vee\Delta(B_T\to\neg B)),\\
\end{array}\]
Therefore, $D$ is definable in terms of $\Delta$.
\begin{proposition}\label{prop.defineD}For all $n\in\mathbb{N}$,
$$\vDash D(A_1,\cdots,A_n;B)\lra\bigwedge_{T\subseteq \{1,\cdots,n\}}(\Delta(B_T\to B)\vee\Delta(B_T\to\neg B)).$$
\end{proposition}

\begin{proof}
Given any $n\in \mathbb{N}$ and any pointed model $(\M,w)$. Suppose, for a contradiction, that $\M,w\vDash D(A_1,\cdots,A_n;B)$ but $\M,w\nvDash\bigwedge_{T\subseteq \{1,\cdots,n\}}(\Delta(B_T\to B)\vee\Delta(B_T\to\neg B))$. Then there exists $T'\subseteq \{1,\cdots,n\}$ such that $\M,w\nvDash \Delta(B_{T'}\to B)$ and $\M,w\nvDash\Delta (B_{T'}\to\neg B)$. It follows that there are $u,v$ with $wRu$ and $wRv$ such that $\M,u\vDash B_{T'}\land \neg B$ and $\M,v\vDash B_{T'}\land B$. Thus $w$ has two successors that agree on the truth value of $A_i$ for each $i\in[1,n]$ but not on the truth value of $B$, which is contrary to the supposition that $\M,w\vDash D(A_1,\cdots,A_n;B)$.

Conversely, suppose that $\M,w\nvDash D(A_1,\cdots,A_n;B)$. Then there are $u,v$ such that $wRu,wRv$ and $(\M,u\vDash A_i\iff \M,v\vDash A_i)$ for each $i\in[1,n]$, but it is {\em not} the case that $(\M,u\vDash B\iff\M,v\vDash B)$. W.l.o.g. we may assume that $\M,u\vDash B$ and $\M,v\nvDash B$. Thus $\M,u\nvDash \neg B$ and $\M,v\vDash\neg B$. Now let $T=\{k\in[1,n]\mid A_k\text{ are true at }u\}$. Recall that $u$ and $v$ agree on $A_i$ for all $i\in[1,n]$. Then it is easy to show that both $u$ and $v$ satisfy every conjunct $B_i$ of $B_T$: according to the construction of $B_T$ as above, if $i\in T$, then $B_i=A_i$ are true at $u$ (and thus also at $v$); if $i\notin T$, then $B_i=\neg A_i$ is also true at $u$ (thus also at $v$).

We have thus proved that $w$ has two successors $u$ and $v$, of which both satisfy $B_T$ but only $u$ satisfies $B$. Thus $\M,u\vDash B_T\to B$ but $\M,v\nvDash B_T\to B$, and hence $\M,w\nvDash\Delta(B_T\to B)$; similarly, $\M,u\nvDash B_T\to \neg B$ but $\M,v\vDash B_T\to\neg B$, and hence $\M,w\nvDash \Delta(B_T\to\neg B)$. Therefore, $\M,w\nvDash\Delta(B_T\to B)\vee\Delta(B_T\to\neg B)$ for some $T\subseteq\{1,\cdots,n\}$, i.e. $\M,w\nvDash\bigwedge_{T\subseteq \{1,\cdots,n\}}(\Delta(B_T\to B)\vee\Delta(B_T\to\neg B))$.
\end{proof}

According to the above analysis and Prop.~\ref{prop.defineD}, we have demonstrated our claim in the opening paragraph of this part.
\begin{proposition}\label{prop.exp-ld-lc}
All $\mathcal{L}_{D^i}$ (where $i\in\mathbb{N}\cup\{\infty\}$) are equally expressive on the class of all models. In particular, $\mathcal{L}_\Delta$ is equally expressive as $\mathcal{L}_D$ on that class. 
\end{proposition}

As a corollary, we obtain
\begin{corollary}
$\mathcal{L}_\Delta$ is equally expressive as $\mathcal{L}_D$ on the class of universal models.
\end{corollary}
It is known that
\begin{proposition}(c.f.~e.g.~\cite[Sec.~3.1]{Fanetal:2015})\label{prop.exp-ln-lc}
$\mathcal{L}_\Delta$ is less expressive than the standard modal logic $\mathcal{L}_\Box$ over the class of all models, the class of $\mathcal{D}$-models, the class of $\mathcal{B}$-models, the class of $4$-models, the class of $5$-models, whereas the two logics are equally expressive over the class of $\mathcal{T}$-models.
\end{proposition}

As a corollary of Prop.~\ref{prop.exp-ld-lc} and Prop.~\ref{prop.exp-ln-lc}, we obtain the following expressivity results.
\begin{corollary}\label{prop.exp-ln-ld}
$\mathcal{L}_D$ is less expressive than the standard modal logic $\mathcal{L}_\Box$ over the class of all models, the class of $\mathcal{D}$-models, the class of $\mathcal{B}$-models, the class of $4$-models, the class of $5$-models, whereas the two logics are equally expressive over the class of $\mathcal{T}$-models.
\end{corollary}

\weg{Since $(\Box(B_T\to B)\vee\Box(B_T\to\neg B))\lra(\Delta(B_T\to B)\vee\Delta(B_T\to\neg B))$, we have
\begin{corollary}\label{prop.BoxdefineD} $D$ is definable in terms of $\Box$. For all $n\in\mathbb{N}$,
$$\vDash D(A_1,\cdots,A_n;B)\lra\bigwedge_{T\subseteq \{1,\cdots,n\}}(\Box(B_T\to B)\vee\Box(B_T\to\neg B)),$$ thus $\mathcal{L}_D$ is a fragment of standard modal logic.
\end{corollary}}

\subsection{Axiomatizing $\mathcal{L}_D$ over various frame classes}\label{sec.axiomatization-ld}

In the above part, we have shown that $\mathcal{L}_D$ is equally expressive as $\mathcal{L}_\Delta$, as the determinacy operator $D$ and the non-contingency operator $\Delta$ are interdefinable with each other. As we know, axiomatizations of $\mathcal{L}_\Delta$ over various classes of frames have been given in the literature, see Sec.~\ref{sec.contingency} for a survey. We may thus obtain the axiomatizations of $\mathcal{L}_D$ from those of $\mathcal{L}_\Delta$ via some translations.

We first define a translation $t_\Delta:\mathcal{L}_\Delta\to\mathcal{L}_D$ as follows:
\[
\begin{array}{lllr}
t_\Delta(p)&=&p&\\
t_\Delta(\neg A)&=&\neg t_\Delta(A)&\\
t_\Delta(A\land B)&=&t_\Delta(A)\land t_\Delta(B)&\\
t_\Delta(\Delta A)&=&D(\epsilon;t_\Delta(A))&(\text{where }\epsilon\text{ is the empty sequence of formulas})\\
\end{array}
\]

The translation $t_D$ from $\mathcal{L}_D$ to $\mathcal{L}_\Delta$ is defined by the following:
\[
\begin{array}{lll}
t_D(p)&=&p\\
t_D(\neg A)&=&\neg t_D(A)\\
t_D(A\land B)&=&t_D(A)\land t_D(B)\\
t_D(D(A_1,\cdots,A_n;B))&=&\bigwedge_{T\subseteq \{1,\cdots,n\}}(\Delta(t_D(B_T)\to t_D(B))\vee\Delta(t_D(B_T)\to\neg t_D(B)))\\
\end{array}
\]
Where $B_T$ is defined as in Sec.~\ref{sec.exp-lcld}.

From the definition of $t_D$, it follows that $\mathcal{L}_D$ is at least exponentially more succinct than $\mathcal{L}_\Delta$.

\medskip

We now construct the proof systems of $\mathcal{L}_D$ over various frame classes. Recall the proof systems for $\mathcal{L}_\Delta$ in Section~\ref{sec.contingency}. Given any proof system $S_\Delta$ mentioned above, we define the system $S_D$ for $\mathcal{L}_D$ to be the extension of $S_\Delta$ plus the following axiom schemas: for each $n\in\mathbb{N}^+$,
\[
\begin{array}{ll}
\texttt{D}_n&D(A_1,\cdots,A_n;B)\lra\bigwedge_{T\subseteq \{1,\cdots,n\}}(\Delta(B_T\to B)\vee\Delta(B_T\to\neg B)).
\end{array}
\]

In the sequel, we will show the completeness of $S_D$ over various classes of frames. The strategy is via a reduction to the completeness of $S_\Delta$: if $S_\Delta$ is sound and strongly complete with respect to a class $F$ of frames, then so is $S_D$.

Firstly, we show that
\begin{proposition}\label{prop.equaltran}
For all $A\in\mathcal{L}_D$, we have $$\vdash_{S_D}A\lra t_D(A).$$
\end{proposition}

\begin{proof}
By induction on $A$. The base case and boolean cases are straightforward from the definition of $t_D$ and induction hypothesis. We only need to consider the case $D(A_1,\cdots,A_n;B)$.

By induction hypothesis, we have $\vdash_{S_D}A_i\lra t_D(A_i)$ for all natural numbers $i\in[1,n]$, and $\vdash_{S_D}B\lra t_D(B)$. Given any $T\subseteq \{1,\cdots,n\}$, according to the construction of $B_T$, we can thus obtain $\vdash_{S_D}B_T\lra t_D(B_T)$. By the axiom $\TAUT$ and the rule $\REKw$, we have
$$\vdash_{S_D}(\bigwedge_{T\subseteq \{1,\cdots,n\}}(\Delta(B_T\to B)\vee\Delta(B_T\to\neg B)))~~~~~~~~~~~~~~~~~~~~~~~~~~~~~~~~~$$
$$~~~~~~~~~~~~~~~~~~~~~~~~~~~~~~~~~~~~~~~~~~\lra(\bigwedge_{T\subseteq \{1,\cdots,n\}}(\Delta(t_D(B_T)\to t_D(B))\vee\Delta(t_D(B_T)\to\neg t_D(B))))$$
Then by the axiom $\texttt{D}_n$ and the definition of $t_D(A_1,\cdots,A_n;B)$, we obtain
$$\vdash_{S_D}D(A_1,\cdots,A_n;B)\lra t_D(A_1,\cdots,A_n;B).$$
Therefore we have now completed the proof of the proposition.
\end{proof}

\begin{theorem}
Let $S_\Delta$ be a proof system for $\mathcal{L}_\Delta$ given in Sec.~\ref{sec.contingency}. If $S_\Delta$ is sound and strongly complete with respect to a class $F$ of frames, then so is $S_D$.
\end{theorem}

\begin{proof}
The soundness of $S_D$ is immediate from the soundness of $S_\Delta$ and Prop.~\ref{prop.defineD}. As for the strong completeness, given any $\Gamma\cup\{\phi\}\subseteq \mathcal{L}_D$ and a class $F$ of frames, suppose that $\Gamma\vDash_FA$, then by Prop.~\ref{prop.equaltran} and the soundness of $S_D$, we have $t_D(\Gamma)\vDash_F t_D(A)$, where $t_D(\Gamma)=\{t_D(B)\mid B\in\Gamma\}$. As $S_\Delta$ is strongly complete with respect to $F$ and $t_D(\Gamma)\cup\{t_D(A)\}\subseteq \mathcal{L}_\Delta$, we obtain $t_D(\Gamma)\vdash_{S_\Delta}t_D(A)$. Since $S_D$ is an extension of $S_\Delta$, every deduction in $S_\Delta$ is also a deduction in $S_D$, whence $t_D(\Gamma)\vdash_{S_D}t_D(A)$. Then using Prop.~\ref{prop.equaltran} again, we derive that $\Gamma\vdash_{S_D}A$. Therefore $S_D$ is strongly complete with respect to $F$.
\end{proof}

By Thm.~\ref{thm.com-lc}, we have a lot of completeness results. In particular,
\begin{corollary}
\begin{enumerate}
\item $\SPLKw+D_n$ is sound and strongly complete with respect to the class of all frames, and also the class of $\mathcal{D}$-frames.
\item $\SPLKwTEuc+\texttt{D}_n$ is sound and strongly complete with respect to the class of $\mathcal{S}5$-frames.
\end{enumerate}
\end{corollary}

We have thus also shown that $\SPLKwTEuc+\texttt{D}_n$ is sound and strongly complete with respect to the class of $\mathcal{S}5$-models, thus every consistent set of such a system is satisfied on an $\mathcal{S}5$-model. By using the generated submodel method (c.f.~e.g.~\cite{blackburnetal:2001}), each such consistent set is also satisfied on a universal model. Therefore, $\SPLKwTEuc+\texttt{D}_n$ is sound and strongly complete over universal models. Thus we also give an alternative axiomatization for $\mathcal{L}_D$ over universal models, in contrast to~\cite[Sec.~7.2]{Gorankoetal:2016}.
\begin{corollary}
$\SPLKwTEuc+\texttt{D}_n$ completely axiomatizes $\mathcal{L}_D$ over the class of universal models.
\end{corollary}

\section{Conclusion and Future work}\label{sec.conclusions}

In this contribution, due to the philosophical importance of the concept of supervenience, and inspired by the notion of supervenience-determined consequence relation and the semantics of agreement operator in the literature, we have proposed a modal logic of supervenience $\mathcal{L}_\Rrightarrow$, which has the dyadic modality of supervenience as a sole primitive modality. We have argued for the naturalness of the semantics of $\Rrightarrow$, in that it corresponds to the supervenience-determined consequence relation, in a similar way that the strict implication corresponds to the inference-determined consequence relation. We have shown that this logic is more expressive than the modal logic of agreement $\mathcal{L}_O$, by introducing a notion of bisimulation for $\mathcal{L}_\Rrightarrow$. We have also presented a sound proof system $\mathbb{LS}$ for $\mathcal{L}_\Rrightarrow$, which, we think, captures the intuition of supervenience. We have also generalized the discussion on the dyadic operator of supervenience into the case where the supervenience operator takes any finitely many formulas as arguments. Moreover, we have shown that non-contingency logic $\mathcal{L}_\Delta$ and propositional logic of determinacy $\mathcal{L}_D$ are equally expressive over the class of all models, thus also equally expressive over the class of universal models; we have also given proof systems of $\mathcal{L}_D$ over various frame classes, whose completeness are shown via a reduction to the completeness proof for the corresponding axiomatizations for $\mathcal{L}_\Delta$, thereby resolving an open research direction listed in~\cite[Sec.~8.2]{Gorankoetal:2016}. Last but not least, we have also established an alternative axiomatization for $\mathcal{L}_D$ over universal models.

Although we do not have the completeness result for $\mathbb{LS}$, we do hope that our study will open up a new research direction for the philosophical notion of supervenience from the viewpoint of philosophical logic. There are lots of future work to be continued, some of which are listed below.

\subsection{Completeness proof for $\mathbb{LS}$}

Recall that $\mathbb{LS}$ has been established in Sec.~\ref{sec.soundsystem} and shown to be sound with respect to the class of all frames. Although we conjecture that $\mathbb{LS}$ is also strongly complete with respect to that class, the completeness result is open. If we adopt the Henkin's canonical model method, then a natural candidate for the canonical model of $\mathbb{LS}$ is as follows:
\begin{definition}[Canonical Model for $\mathbb{LS}$]
\begin{itemize}\
\item $W^c=\{w\mid w\text{ is a maximal consistent set for }\mathbb{LS}\}$
\item $S^c$ is defined such that for every $w\in W^c$, we have
\[\begin{array}{lcl}
S^c_wuv&\iff& \text{for all }A,B,\text{ if }A\vartriangleleft B\in w\text{ and }(A\in u\iff A\in v),\\
&&\text{then }(B\in u\iff B\in v).
\end{array}\]
\item $V^c(p)=\{w\in W^c\mid p\in w\}$.
\end{itemize}
\end{definition}


Then it is straightforward to show the `if' part in the Truth Lemma, i.e., for all $A\in L_\vartriangleleft$, for all $w\in W^c$, we have
$\M^c,w\vDash A\iff A\in w.$ In order to obtain the `only if' part, we need (and only need) to show: 
\begin{proposition}
let $C\vartriangleleft D\notin w$ with $w\in W^c$. Then there exist $u,v\in W^c$ such that $S^c_wuv$ and $(C\in u\iff C\in v)$ but $D\in u$ and $D\notin v$.
\end{proposition}
We do not know how to show this proposition in the current stage, we even do not whether it holds. We leave it for future work. If the above proposition is shown, then we get not only the completeness result over the class of all frames, but also that over the class of reflexive and symmetric frames, since $S_w^c$ is reflexive and symmetric for all $w\in W^c$. Besides, we can also investigate the extensions of $\mathbb{LS}$ on special frame classes, specifically, the axiomatizations over universal frames/models.

\subsection{Comparing the expressive powers of $\mathcal{L}_{\Rrightarrow^n}$ and $\mathcal{L}_{\Rrightarrow^{n+1}}$}

We also leave Conjecture~\ref{conj.exp-ls} for future work, which states that for all $n\in\mathbb{N}$, $\mathcal{L}_{\vartriangleleft^{n+1}}$ is more expressive than $\mathcal{L}_{\vartriangleleft^n}$. The proof strategy should be similar to that used in showing that $\mathcal{L}_\Rrightarrow(=\mathcal{L}_{\vartriangleleft^1})$ is more expressive than $\mathcal{L}_O(=\mathcal{L}_{\vartriangleleft^0})$. In detail, first introduce a bisimulation notion for $\mathcal{L}_{\vartriangleleft^n}$, such that any such bisimilar models are indistinguishable by any $\mathcal{L}_{\vartriangleleft^n}$ formulas; then we construct two such bisimilar models which can be distinguished by an $\mathcal{L}_{\vartriangleleft^{n+1}}$ formula. Furthermore, the bisimulation notion for $\mathcal{L}_{\vartriangleleft}$ should be an adaption of that for $\mathcal{L}_{O}$, and both of the required bisimilar models should have at least $2^{n+1}$ worlds, just as we need at least $4=2^2$ worlds in showing Prop.~\ref{prop.moreexp}. For the details, refer to Sec.~\ref{sec.exp-ls-lo}. Once we do this, we can obtain an expressive hierarchy from the weakest $\mathcal{L}_O$ to the strongest $\mathcal{L}_{\vartriangleleft^\infty}$.

\subsection{Relativized agreement operator}

It is worth mentioning that the semantics of $A\vartriangleleft B$ can be rephrased as follows:
\[\begin{array}{|lcl|}
\hline
\M,w\vDash A\vartriangleleft B&\iff&\text{for all }u,v\in\M\text{ such that }S_wuv, \text{ if }(\M,u\vDash A \text{ and } \M,v\vDash A),\\
&&\text{ then }(\M,u\vDash B\iff \M,v\vDash B), \text{ and},\\
&&\text{for all }u,v\in\M\text{ such that }S_wuv, \text{ if }(\M,u\vDash \neg A \text{ and } \M,v\vDash\neg A),\\
&&\text{ then }(\M,u\vDash B\iff \M,v\vDash B).\\
\hline
\end{array}
\]

If we define a relativized agreement operator $O(\cdot,\cdot)$ (denoted $O^r$), as
\[\begin{array}{lcl}
\M,w\vDash O(A,B)&\iff&\text{for all }u,v\in\M\text{ such that }S_wuv, \text{ if }(\M,u\vDash A \text{ and } \M,v\vDash A),\\
&&\text{then }(\M,u\vDash B\iff \M,v\vDash B).\\
\end{array}\]

Roughly, $O(A,B)$ means that the truth value of $B$ is fixed under the condition that $A$. Then $(A\vartriangleleft B)$ is logically equivalent to $O(A,B)\land O(\neg A,B)$, and hence we have that $\mathcal{L}_{O^r}$ is at least as expressive as $L_\vartriangleleft$, and that also more expressive than the unrelativized version $\mathcal{L}_O$, where $\mathcal{L}_{O^r}$ denote the extension of propositional logic with the relativized agreement operator $O^r$. We conjecture that $\mathcal{L}_{O^r}$ is more expressive than $\mathcal{L}_\vartriangleleft$, since even the simplest $\mathcal{L}_{O^r}$ formula $O(p,q)$ seems to be not expressible in $\mathcal{L}_\vartriangleleft$. We left it for future work. Another research worth investigating is, of course, to axiomatize $\mathcal{L}_{O^r}$ over various frame classes.

\subsection{Relativized supervenience operator}

We can also generalize the supervenience operator into a relativized version. In detail, we introduce a relativized supervenience operator $\cdot\Rrightarrow^\cdot\cdot$ (denoted $\Rrightarrow^r$) into propositional logic, and we denote the resulting language as $\mathcal{L}_{\vartriangleleft^r}$.

Formally, the new operator is defined as follows:
\[\begin{array}{lcl}
\M,w\vDash A\Rrightarrow^C B&\iff&\text{for all }u,v\in\M\text{ such that }S_wuv\text{ and }\M,u\vDash C\text{ and }\M,v\vDash C, \\
&&\text{if }(\M,u\vDash A \iff \M,v\vDash A),\\
&&\text{then }(\M,u\vDash B\iff \M,v\vDash B).\\
\end{array}\]
Intuitively, $A\Rrightarrow^C B$ says that under the condition $C$, $B$ supervenes on $A$. Now the dyadic supervenience operator $\Rrightarrow$ is definable in terms of $\Rrightarrow^r$, as $(A\Rrightarrow B)=_{df}(A\Rrightarrow^T B)$. Similar to the case for $\mathcal{L}_{O^r}$, we conjecture that $\mathcal{L}_{\vartriangleleft^r}$ is more expressive than $\mathcal{L}_{\vartriangleleft}$, for $p\vartriangleleft^r q$ seems to be undefinable with any $\mathcal{L}_{\vartriangleleft}$ formulas.

Moreover, we can further compare the expressive powers of $\mathcal{L}_{\vartriangleleft^r}$ and $\mathcal{L}_{O^r}$. Note that $\mathcal{L}_{\vartriangleleft^r}$ is at least as expressive as $\mathcal{L}_{O^r}$, due to the logical equivalence $\vDash O(A,B)\lra (\top \Rrightarrow^AB)$ (or $\vDash O(A,B)\lra (\bot \Rrightarrow^AB)$). However, the comparison in expressivity is different from the case between $\mathcal{L}_{\vartriangleleft}$ and $\mathcal{L}_{O}$. Recall that $\mathcal{L}_{\vartriangleleft}$ is more expressive than $\mathcal{L}_{O}$ (Prop.~\ref{prop.moreexp}). Turn to $\mathcal{L}_{\vartriangleleft^r}$ and $\mathcal{L}_{O^r}$, we observe that $A\Rrightarrow^C B$ is logically equivalent to an $\mathcal{L}_{O^r}$ formula $O(A\land C,B)\land O(\neg A\land C,B)$. Therefore, $\mathcal{L}_{\vartriangleleft^r}$ and $\mathcal{L}_{O^r}$ are equally expressive. Thus once we have the comparison results involving one of the two logics, we have also the same comparison results for the other logic. And also, once we have the axiomatizations for one of the two logics, we have also the corresponding axiomatizations for the other logic, via a reduction of completeness results.

\subsection{Characterizing the supervenience-determined consequence relation}

Traditionally, Tarski's consequence relation, called the inference-determined consequence relation in Humberstone~\cite{Humberstone:1993}, has been used as a standard notion of logical consequence, and various proof systems have been established to characterize this notion.

Instead of the standard notion of Tarski's consequence relation, we could also use the notion of supervenience-determined consequence relation as logical consequence, and present Hilbert- (and Gentzen-, etc.) style proof systems to characterize this kind of logical consequence. 

\subsection{Combing the notions of ceteris paribus and supervenience}

Ceteris paribus, meaning ``all else being equal'' or ``(all) others being held constant'', is a very common term in our daily life. The term has been widely used in defining the laws in special sciences, see~e.g.~\cite{sep-ceteris-paribus} for a survey. This notion has also been applied to analyse the notion of preference~\cite{Doyleetal:1994,vanbenthemetal:2009}, counterfactual reasoning~\cite{Girardetal:2016}, agency and games~\cite{Grossietal:2013,Grossietal:2015}, Fitch's paradox~\cite{Proiettietal:2010}, and the future contingents problem~\cite{Proietti:2009}, etc..

It may be interesting to combine ceteris paribus and supervenience, since in that case we can naturally express the statements such as ``Ceteris paribus, $B$ supervenes on $A$'', or more general, ``Ceteris paribus, $B$ supervenes on $A_1,\cdots,A_n$''. We can then compare the new logics with our modal logic of supervenience $\mathcal{L}_\vartriangleleft$, in both expressivity and axiomatizations.

\section*{Acknowledgements}

This research is funded by China Postdoctoral Science Foundation [Grant number: 2016M590061]. The author would like to thank Lloyd Humberstone for discussions on an earlier version of this paper, and also for providing several references.

\bibliographystyle{plain}
\bibliography{biblio2016}

\end{document}